\documentclass[12pt,a4paper]{article}
\usepackage{latexsym,amsfonts,amsmath,graphics,amsthm}
\usepackage{epsfig}
\usepackage{hyperref}

\newtheorem{theorem}{Theorem}
\newtheorem{lemma}{Lemma}

\newtheorem{proposition}{Proposition}

\newcommand{\norm}[1]{\left\Vert#1\right\Vert}

\newcommand{\bsa}{\boldsymbol{a}}

\newcommand{\bsl}{\boldsymbol{l}}
\newcommand{\bsv}{\boldsymbol{v}}
\newcommand{\bsw}{\boldsymbol{w}}
\newcommand{\bsb}{\boldsymbol{b}}
\newcommand{\bsx}{\boldsymbol{x}}
\newcommand{\bsh}{\boldsymbol{h}}

\newcommand{\bst}{\boldsymbol{t}}

\newcommand{\bsy}{\boldsymbol{y}}

\newcommand{\cP}{{\cal P}}
\newcommand{\cG}{{\cal G}}

\newcommand{\e}{{\varepsilon}}

\newcommand{\uu}{\mathfrak{u}}

\newcommand{\icomp}{\mathtt{i}}
\newcommand{\bszero}{\boldsymbol{0}}
\newcommand{\rd}{\,\mathrm{d}}
\newcommand{\NN}{\mathbb{N}}
\newcommand{\ZZ}{\mathbb{Z}}

\newcommand{\RR}{\mathbb{R}}

\newcommand{\bL}{\mathbb{L}}
\newcommand{\nat}{\NN}

\renewcommand{\pmod}[1]{\,(\bmod\,#1)}

\newcommand{\il}{\left<}
\newcommand{\ir}{\right>}
\def\qed{\hfill$\Box$}
\newcommand{\abs}[1]{\left\vert#1\right\vert}
\def\calV{{\cal V}}
\newcommand{\To}{\rightarrow}

\allowdisplaybreaks

\begin{document}

\title{$\bL_{\infty}$-approximation in Korobov spaces\\
with Exponential Weights}

\author{Peter Kritzer\thanks{P.~Kritzer
is supported by the Austrian Science Fund (FWF),
Project F5506-N26, which is part of the Special Research Program ``Quasi-Monte Carlo Methods:
Theory and Applications''.},
 Friedrich Pillichshammer\thanks{F.~Pillichshammer
is supported by the Austrian Science Fund (FWF),
Project F5509-N26, 
which is part of the Special Research Program ``Quasi-Monte Carlo Methods:
Theory and Applications''.},
Henryk Wo\'zniakowski\thanks{H. Wo\'zniakowski is partially supported
by the National Science Centre, Poland, based on the decision
DEC-2013/09/B/ST1/04275.}}
\maketitle

\begin{abstract}
We study multivariate $\bL_{\infty}$-approximation for a weighted Korobov space
of periodic functions for which
the Fourier coefficients decay exponentially fast.
The weights are defined, in particular, 
in terms of two sequences $\bsa=\{a_j\}$
and $\bsb=\{b_j\}$ of positive real numbers 
bounded away from zero. We study 
the minimal worst-case error $e^{\bL_{\infty}\mathrm{-app},\Lambda}(n,s)$
of all algorithms that use $n$ information evaluations from 
a class~$\Lambda$ in the $s$-variate case. 
We consider two classes $\Lambda$ in this paper:
the class $\Lambda^{{\rm all}}$ of all linear functionals 
and the class $\Lambda^{{\rm std}}$ of only function evaluations.

We study exponential convergence 
of the minimal worst-case error, which means that
$e^{\bL_{\infty}\mathrm{-app},\Lambda}(n,s)$ converges to zero 
exponentially fast with increasing $n$. Furthermore,
we consider how the error depends on the dimension $s$. 
To this end, we define the notions of $\kappa$-EC-weak, 
EC-polynomial and EC-strong polynomial tractability, where EC stands for ``exponential convergence''. 
In particular, EC-polynomial tractability means that we need 
a polynomial number of information evaluations in $s$ and
$1+\log\,\e^{-1}$ to compute an $\e$-approximation.
We derive necessary and sufficient 
conditions on the sequences $\bsa$ and $\bsb$ 
for obtaining exponential error convergence, and also for obtaining the various 
notions of tractability.
The results are the same for both classes $\Lambda$. 

$\bL_2$-approximation for functions from 
the same function space has been considered in \cite{DKPW13}. 
It is surprising that most results for $\bL_{\infty}$-approximation 
coincide with their counterparts for $\bL_2$-approximation. 
This allows us to deduce also results 
for $\bL_p$-approximation for $p \in [2,\infty]$.
\end{abstract}

\centerline{\begin{minipage}[hc]{130mm}{
}
\end{minipage}}

\section{Introduction}

Function approximation is a topic addressed
in a huge number of papers and monographs.
We study the problem of multivariate $\bL_{\infty}$-approximation
of functions which belong 
to a special class of one-periodic functions defined on $[0,1]^s$.
Here, $s\in\NN:=\{1,2,\dots\}$ and our emphasis is on large $s$.
These functions belong to a weighted Korobov space
whose elements share the property that their Fourier
coefficients decay exponentially fast.

Korobov spaces are special types of reproducing kernel Hilbert spaces
$H(K_s)$ with a reproducing kernel of the form
\[K_s(\bsx,\bsy)=\sum_{\bsh\in\ZZ^s}
\rho (\bsh) \exp (2\pi\icomp \bsh\cdot(\bsx-\bsy))\ \ \ \ \
\mbox{for all $\bsx,\bsy\in[0,1]^s$}\]
for some function $\rho:\ZZ^s\to \RR_+$.

We approximate functions by algorithms that use $n$ information
evaluations, where we allow information evaluations
from the class $\Lambda^{\rm{all}}$ of
all continuous linear functionals or, alternatively, from
the narrower class
$\Lambda^{\rm{std}}$ of standard information
which consists of only function evaluations.
The quality of our algorithms is measured by
the worst-case approximation error, i.e.,
by the largest error over the unit ball of the function space.

For large~$s$, it is crucial to study how the errors of algorithms
depend not only on~$n$ but also on~$s$.
The information complexity $n^{\bL_{\infty}\mathrm{-app},\Lambda}(\e,s)$
is the minimal number~$n$ for which
there exists an
algorithm using $n$ information evaluations from the class
$\Lambda\in\{\Lambda^{\rm{all}},\Lambda^{\rm{std}}\}$ with
an error of at most $\e$ times a constant. If this constant is 1, then
we speak about the absolute error criterion.
If the constant is equal to the initial error, i.e., 
the error without using any information evaluations,
we speak about the normalized
error criterion. The information complexity
is proportional to the minimal cost of computing an $\e$-approximation
since linear algorithms are optimal and their cost is proportional to
$n^{\bL_{\infty}\mathrm{-app},\Lambda}(\e,s)$.

\medskip

In many papers, as for example \cite{KSW06, KWW08, KWW09a,
KWW09b, LH03, ZKH09}, and also in the recent
trilogy \cite{NW08}--\cite{NW12},
Korobov spaces are studied for functions $\rho$ of the form
\[\rho_{\mathrm{pol}} (\bsh)=\rho_{\mathrm{pol}}(h_1,h_2,\ldots,h_s)=
\prod_{j=1}^{s} \rho_{\mathrm{pol}} (h_j),\ \ \
\rho_{\mathrm{pol}}(h_j)=\begin{cases} 1 &\mbox{if $h_j=0$},\\
                        \gamma_j / \abs{h_j}^\alpha & \mbox{if $h_j\neq 0$},
                                                         \end{cases}
\]
where $\alpha>1$ is a smoothness parameter for the elements of
the Korobov space (the number of derivatives of
the functions is roughly $\alpha/2$).
This means that $\rho_{\mathrm{pol}}(\bsh)$ decays
polynomially in $\bsh$. The function $\rho_{\mathrm{pol}}$
also depends on weights $\gamma_j$ which model
the influence of the different variables and
groups of variables
of the problem. For this choice of $\rho_{\mathrm{pol}}$, it can
be shown that one can achieve polynomial error convergence
for $\bL_2$- and $\bL_{\infty}$-approximation and,
under suitable conditions on the weights, also
avoid a curse of dimensionality and achieve different types of
tractability.
By tractability, we mean that
the information complexity 
does neither depend exponentially on $s$ nor on $\e^{-1}$. In particular,
we speak of polynomial tractability if the information complexity
depends at most polynomially
on $s$ and $\e^{-1}$ and of strong polynomial tractability
if it depends polynomially on $\e^{-1}$ and not on
$s$, see \cite{KSW06, KWW08, KWW09a, KWW09b,NSW04}
as well as \cite{NW08}--\cite{NW12} for further
details.  We stress that the results for $\bL_2$- and
$\bL_{\infty}$-approximation for Korobov spaces
based on $\rho_{\mathrm{pol}}$ are not the same.
Indeed, let
$$
\delta_0:=\inf\left\{\delta \ge0:\, 
\sum_{j=1}^\infty \gamma_j^{\delta} <\infty\right\}
\ \ \ \mbox{and}\ \ \ \alpha_0:=\min (\alpha, \delta_0^{-1}).
$$
Then the best rate of error convergence with 
strong polynomial tractability
for a Korobov space based on $\rho_{\mathrm{pol}}$ is, 
if we allow information from $\Lambda^{\rm all}$,
of order $\alpha_0/2$ for $\bL_2$-approximation, 
and $(\alpha_0 -1)/2$, if $\alpha_0>1$, 
for  $\bL_{\infty}$-approximation.
The convergence rates for $\Lambda^{\rm std}$ are 
not known exactly and the known upper bounds are slightly weaker 
than for the class $\Lambda^{\rm all}$. 

\medskip

For the Korobov spaces considered in the present paper, 
we choose $\rho$ as $\rho_{\mathrm{exp}} (\bsh)$ 
which decays exponentially in $\bsh$,
and again $\rho_{\mathrm{exp}}$ depends on weights expressed
by two sequences of positive real numbers $\bsa$ and $\bsb$, 
which model the influence of the variables of
the problem. For this choice of $\rho_{\mathrm{exp}}$, 
we obtain exponential error convergence
instead of polynomial error convergence. To be more precise, let
$e^{\bL_{\infty}\mathrm{-app},\Lambda}(n,s)$
be the minimal worst-case error
among all algorithms that use
$n$ information evaluations from
a permissible class~$\Lambda$ in the $s$-variate case.
By \emph{exponential convergence}
of the $n$th minimal approximation error
we mean that
$$
e^{\bL_{\infty}\mathrm{-app},\Lambda}(n,s)\le C(s)\,
q^{\,(n/C_1(s))^{\,p(s)}}\ \ \ \ \mbox{for all}\ \ \ \ n,s\in\NN,
$$
where, $q\in(0,1)$ is independent of $s$,
whereas $C,C_1,$ and $p$ are allowed to be dependent on~$s$.
We have \emph{uniform exponential convergence}
if $p$ can be chosen independently of $s$. 

Under suitable conditions on the weight sequences $\bsa$ and $\bsb$, 
we achieve stronger notions of tractability than for the case of 
polynomial error convergence, which is
then referred to as Exponential Convergence-tractability 
(or, for short, EC-tractability). Roughly speaking, 
EC-tractability is defined similarly to the standard notions of
tractability, but we replace $\e^{-1}$ by $1 +\log\,\e^{-1}$.

The case of $\bL_2$-approximation for $\rho (\bsh)$ depending exponentially
on $\bsh$ was dealt with in the recent paper \cite{DKPW13}. 
The results there can be also obtained as special cases of 
a more general approach presented in the paper \cite{IKPW15}. 

For the case of $\bL_{\infty}$-approximation, which is considered in the
present paper, it turns out that most of 
the results are the same as for $\bL_2$-approximation. 
Surprising as this may seem,
the reason for the similarities between $\bL_2$- 
and $\bL_{\infty}$-approximation
may lie in the expression of the worst-case error in 
terms of the ordered eigenvalues $\lambda_1,\lambda_2,\ldots$ of
a certain operator $W_s :H(K)\rightarrow H(K)$, see below. 
For $\bL_2$-approximation, the minimal error if we use $n$ evaluations 
from $\Lambda^{\rm all}$ is 
$\sqrt{\lambda_{n+1}}$,
whereas the minimal error for $\bL_{\infty}$-approximation 
is  $\sqrt{\sum_{k=n+1}^\infty \lambda_k}$. 
For the case
of the spaces considered in this paper, the eigenvalues $\lambda_k$ 
depend exponentially on $k$, which means that $\lambda_n$ and
$\sum_{k=n+1}^\infty \lambda_k$ behave similarly, 
which suggests that the errors for $\bL_2$-approximation 
and $\bL_{\infty}$-approximation should
also have similar properties. Moreover, 
as we shall also show in the present paper, there are 
no differences in the results 
between the class $\Lambda^{\rm all}$ and the
class $\Lambda^{\rm std}$, and no difference 
between the absolute and the normalized error criterion.
However, for one concept of tractability there is a difference in the
results between $\bL_2$-approximation 
and $\bL_{\infty}$-approximation.

\medskip

The rest of the paper is structured as follows. 
In Section \ref{secKor}, we introduce the weighted Korobov space 
considered in this paper,
and in Section \ref{seclinfty} we define precisely what we mean 
by exponential error convergence and
by various notions of Exponential Convergence-tractability. 
Our main result is stated in Section \ref{mainresult}.
Furthermore, in Section \ref{relL2}, we outline relations 
of the $\bL_{\infty}$-approximation
problem to the $\bL_2$-approximation problem. 
After some preliminary observations in Section \ref{secprelA}, 
we then outline our main results 
in Sections \ref{secexp}--\ref{secpt}. In Section \ref{secend}, 
we summarize and compare our results on $\bL_\infty$-approximation to previous 
results on $\bL_2$-approximation, and, in the final 
Section~\ref{rem_lp_approx}, we give some remarks on $\bL_p$-approximation.

\section{The Korobov space $H(K_{s,\bsa,\bsb})$}\label{secKor}

The Korobov space $H(K_{s,\bsa,\bsb})$
discussed in this section is a reproducing kernel Hilbert space.
For general information on reproducing kernel Hilbert
spaces we refer to~\cite{Aron}.

Let $\bsa=\{a_j\}_{j \ge 1}$ and
$\bsb=\{b_j\}_{j \ge 1}$ be two sequences of real positive weights
such that
\begin{equation}\label{coefficients}
b_\ast:=\inf_j b_j>0\ \ \ \ \ \mbox{and}\ \ \ \ \
a_\ast:=\inf_j a_j>0.
\end{equation}
Throughout the paper we additionally assume that
$$
a_\ast=a_1 \le a_2 \le a_3 \le \ldots.
$$
Fix $\omega\in(0,1)$ and denote
\[
\omega_{\bsh}=\omega^{\sum_{j=1}^{s}a_j \abs{h_j}^{b_j}}
\qquad\mbox{for all}\qquad \bsh=(h_1,h_2,\dots,h_s)\in\ZZ^s.
\]
We consider a Korobov space of complex-valued one-periodic
functions defined on $[0,1]^s$ with a reproducing
kernel of the form
\begin{equation*}
K_{s,\bsa,\bsb}(\bsx,\bsy) =
\sum_{\bsh \in \ZZ^s} \omega_{\bsh}\,
\exp\big(2\pi
\icomp \bsh \cdot(\bsx-\bsy)\big) \ \ \ \mbox{for all}\ \ \
\bsx,\bsy\in[0,1]^s
\end{equation*}
with the usual dot product
$$
\bsh\cdot(\bsx-\bsy)=\sum_{j=1}^sh_j(x_j-y_j),
$$
where
$h_j,x_j,y_j$ are the $j$th components of the vectors
$\bsh,\bsx,\bsy$, respectively, and $\icomp=\sqrt{-1}$.

The kernel $K_{s,\bsa,\bsb}$ is well
defined since
\begin{equation}\label{finitesumweights}
|K_{s,\bsa,\bsb}(\bsx,\bsy)|\le
K_{s,\bsa,\bsb}(\bsx,\bsx)= \sum_{\bsh \in \ZZ^s} \omega_{\bsh}=
\prod_{j=1}^s \left(1+2 \sum_{h=1}^{\infty} \omega^{a_j h^{b_j}}\right)<\infty.
\end{equation}
The last series is indeed finite since
$$
\sum_{h=1}^{\infty} \omega^{a_j h^{b_j}}\le
\sum_{h=1}^{\infty} \omega^{a_\ast h^{b_\ast}}<\infty,
$$
and both $a_\ast$ and $b_\ast$ are assumed to be strictly greater than zero.

The Korobov space with reproducing kernel
$K_{s,\bsa,\bsb}$ is a reproducing kernel Hilbert space
and is denoted by $H(K_{s,\bsa,\bsb})$.
We suppress the dependence on $\omega$ in the notation
since $\omega$ will be fixed throughout the paper and $\bsa$ and
$\bsb$ will be varied.

Clearly, functions from $H(K_{s,\bsa,\bsb})$ are infinitely many
times differentiable, see \cite{DLPW11},
and, if $b_{\ast}\ge 1$ they are also analytic 
as shown in \cite[Proposition~2]{DKPW13}
\footnote{The assumption $b_{\ast}\ge1$ is not explicit
but it is needed in the proof of \cite[last line of p.27]{DKPW13}.}.

For $f\in H(K_{s,\bsa,\bsb})$ we have
$$
f(\bsx)=\sum_{\bsh\in\ZZ^s} \widehat f(\bsh)\,\exp(2\pi \icomp
\bsh \cdot\bsx) \ \ \ \mbox{for all}\ \ \ \bsx\in [0,1]^s,
$$
where $\widehat{f}(\bsh) =
\int_{[0,1]^s} f(\bsx) \exp(-2 \pi \icomp \bsh \cdot \bsx) \rd \bsx$
is the $\bsh$th Fourier coefficient.
The inner product of $f$ and $g$ from $H(K_{s,\bsa,\bsb})$ is given by
$$
\il f,g\ir_{H(K_{s,\bsa,\bsb})}=\sum_{\bsh\in \ZZ^s}\widehat f(\bsh)\,
\overline{\widehat g(\bsh)}\, \omega_{\bsh}^{-1},
$$ where $\overline{z}$ means the complex conjugate of 
$z \in \mathbb{C}$, and the norm of $f$ from $H(K_{s,\bsa,\bsb})$ by
$$
\|f\|_{H(K_{s,\bsa,\bsb})}=\left(\sum_{\bsh\in \ZZ^s}|\widehat
f(\bsh)|^2\omega_{\bsh}^{-1}\right)^{1/2}<\infty.
$$

Integration of functions from $H(K_{s,\bsa,\bsb})$ was already
considered in \cite{KPW12} and, in the case $a_j=b_j=1$ for all $j \in
\NN$, also in \cite{DLPW11}. Furthermore, multivariate
approximation of functions from
$H(K_{s,\bsa,\bsb})$ in the $\bL_2$ norm was considered in the recent
papers \cite{DKPW13,IKPW15}. A survey of these results can be found in
\cite{KPW14}.
In the present paper we consider
the problem of multivariate approximation in the $\bL_{\infty}$ norm
which we shortly call $\bL_{\infty}$-approximation.

\section{$\bL_{\infty}$-approximation}\label{seclinfty}

In this section we consider $\bL_{\infty}$-approximation of functions from
$H(K_{s,\bsa,\bsb})$.
This problem is defined as an approximation
of the embedding from the Korobov space $H(K_{s,\bsa,\bsb})$ to the space
$\bL_{\infty}([0,1]^s)$, i.e.,
$$
{\rm EMB}_{s,\infty}:H(K_{s,\bsa,\bsb}) 
\rightarrow \bL_{\infty}([0,1]^s)\ \ \ \ \
\mbox{given by}\ \ \ \ \ {\rm EMB}_{s,\infty}(f)=f.
$$
This embedding is continuous since
for $f\in H(K_{s,\bsa,\bsb})$ we have
$f(x)=\il f,K_{s,\bsa,\bsb}(\cdot,x)\ir_{H(K_{s,\bsa,\bsb})}
$ and
\begin{eqnarray*}
\|{\rm EMB}_{s,\infty}(f)\|_{\bL_{\infty}([0,1]^s)}&=&
\|f\|_{\bL_{\infty}([0,1]^s)}=\sup_{x\in[0,1]^s}|f(x)|
=\sup_{x\in[0,1]^s}|\il f,K_{s,\bsa,\bsb}(\cdot,x)\ir_{H(K_{s,\bsa,\bsb})}|\\
&\le& \|f\|_{H(K_{s,\bsa,\bsb})}\,
\sup_{x\in[0,1]^s}\sqrt{K_{s,\bsa,\bsb}(x,x)}\\
&=&
\|f\|_{H(K_{s,\bsa,\bsb})}\,
\prod_{j=1}^s \left(1+2 \sum_{h=1}^{\infty}
\omega^{a_j h^{b_j}}\right)^{1/2}.
\end{eqnarray*}
Here, we use the supremum instead of the essential supremum
since $f$ is continuous. Furthermore, the last inequality is sharp
for $f=K_{s,\bsa,\bsb}(\cdot,x)$ for any $x\in[0,1]^s$.
This proves that
$$
\|{\rm EMB}_{s,\infty}\|=
\prod_{j=1}^s \left(1+2 \sum_{h=1}^{\infty}
\omega^{a_j h^{b_j}}\right)^{1/2}.
$$

Without loss of generality, see e.g., \cite{TWW88},
we approximate ${\rm EMB}_{s,\infty}$
by  linear algorithms~$A_{n,s}$ of the form
\begin{equation}\label{linalg}
  A_{n,s}(f) = \sum_{k=1}^{n}\alpha_k L_k(f)\ \ \ \
\mbox{for} \ \ \ \ \ f \in H(K_{s,\bsa,\bsb}),
\end{equation}
where each $\alpha_k$ is a function from $\bL_{\infty}([0,1]^{s})$ and
each $L_k$ is a continuous linear functional defined on $H(K_{s,\bsa,\bsb})$
from a permissible class $\Lambda$ of information. We consider two
classes:

\begin{itemize}
\item $\Lambda=\Lambda^{\mathrm{all}}$ , the class of all continuous
linear functionals defined on $H(K_{s,\bsa,\bsb})$.
Since  $H(K_{s,\bsa,\bsb})$ is a Hilbert space,
for every $L_k\in \Lambda^{\mathrm{all}}$ there exists a
function $f_k$ from $H(K_{s,\bsa,\bsb})$ such that
$L_k(f)=\il f,f_k\ir_{H(K_{s,\bsa,\bsb})}$ for all $f\in
H(K_{s,\bsa,\bsb})$.
\item $\Lambda=\Lambda^{\mathrm{std}}$, the class of
standard information consisting only of function evaluations.
That is, $L_k\in\Lambda^{\mathrm{std}}$ iff there exists
$\bsx_k\in[0,1]^{s}$ such that $L_k(f)=f(\bsx_k)$ for all $f\in
H(K_{s,\bsa,\bsb})$.
\end{itemize}

Since $H(K_{s,\bsa,\bsb})$ is a reproducing kernel Hilbert space,
function evaluations are continuous linear functionals and therefore
$\Lambda^{\mathrm{std}}\subseteq \Lambda^{\mathrm{all}}$. More
precisely,
$$
L_k(f)=f(\bsx_k)=\il
f,K_{s,\bsa,\bsb}(\cdot,\bsx_k)\ir_{H(K_{s,\bsa,\bsb})}$$
and
$$\|L_k\|=\|K_{s,\bsa,\bsb}(\cdot,\bsx_k)\|_{H(K_{s,\bsa,\bsb})}=
\sqrt{K_{s,\bsa,\bsb}(\bsx_k,\bsx_k)}=\prod_{j=1}^s \left(1+2 \sum_{h=1}^{\infty} \omega^{a_j h^{b_j}}\right)^{1/2}.
$$

The \textit{worst-case error} of the algorithm \eqref{linalg} is
defined as
\[
  e^{\bL_{\infty}\mathrm{-app}}(H(K_{s,\bsa,\bsb}),A_{n,s})
  :=  \sup_{f \in H(K_{s,\bsa,\bsb}) \atop \norm{f}_{H(K_{s,\bsa,\bsb})}\le 1}
  \norm{f-A_{n,s}(f)}_{\bL_{\infty}([0,1]^s)},
\]
where $\norm{f-A_{n,s}(f)}_{\bL_{\infty}([0,1]^s)}$ is
defined in terms of the essential supremum.

Let $e^{\bL_{\infty}\mathrm{-app},\Lambda}(n,s)$ be the $n$th minimal
worst-case error,
$$
e^{\bL_{\infty}\mathrm{-app},\Lambda}(n,s) = \inf_{A_{n,s}}
e^{\bL_{\infty}\mathrm{-app}}(H(K_{s,\bsa,\bsb}),A_{n,s}),
$$
where the infimum is taken
over all linear algorithms $A_{n,s}$ 
of the form \eqref{linalg} using $n$ information evaluations
from the class $\Lambda$.
For $n=0$ the best we can do is to approximate $f$ by zero,
and the initial error is

\begin{equation}\label{eqinitial}
e^{\bL_{\infty}\mathrm{-app}}(0,s) = \|{\rm EMB}_{s,\infty}\|=
\prod_{j=1}^s \left(1+2 \sum_{h=1}^{\infty} \omega^{a_j h^{b_j}}\right)^{1/2}.
\end{equation}
Note that the initial error may be arbitrarily large for large $s$.
For example, take $a_j=b_j=1$ for all $j\ge 1$. Then
$$
\|{\rm EMB}_{s,\infty}\|=\left(1+\frac{2\omega}{1-\omega}\right)^{s/2}
$$
is exponentially large in $s$.
This means that $\bL_{\infty}$-approximation may be  {\it not} properly
normalized. On the other hand, if
$\sum_{j,h=1}^\infty\omega^{a_jh^{b_j}}<\infty$
then $\|{\rm EMB}_{s,\infty}\|$
is of order $1$ for all~$s$, and
$\bL_{\infty}$-approximation is properly normalized. In particular, this
holds for $a_j=j$ and $b_j=1$ since then
$\sum_{j,h=1}^\infty\omega^{a_jh^{b_j}}<(1-\omega)^{-2}$
and $\|{\rm EMB}_{s,\infty}\|\le \exp((1-\omega)^{-2})$.
\medskip

We study \emph{exponential convergence} in this paper, which is
abbreviated as EXP.
As in~\cite{DKPW13,KPW12,KPW14},
this means that
there exist a number $q\in(0,1)$ and
functions $p,C,C_1:\NN\to (0,\infty)$ such that
\begin{equation}\label{exrate}
e^{\bL_{\infty}\mathrm{-app},\Lambda}(n,s)\le C(s)\, q^{\,(n/C_1(s))^{\,p(s)}}
\ \ \ \ \ \mbox{for
all} \ \ \ \ \ n\in \NN.
\end{equation}
If \eqref{exrate} holds we would like to find the largest possible
rate $p(s)$ of exponential convergence
defined as
\begin{equation}\label{exratemax}
p^*(s)=\sup\{\,p(s)\, :\ \ p(s)\ \ \mbox{satisfies \eqref{exrate}}\,\}.
\end{equation}

\emph{Uniform exponential convergence}, abbreviated as UEXP, means that
the function $p$
in \eqref{exrate} can be taken as a constant function, i.e.,
$p(s)=p>0$ for all $s\in\NN$. Similarly, let
$$
p^*=\sup\{\,p\, :\ \ p(s)=p>0\ \ \mbox{satisfies \eqref{exrate} for all
$s\in\NN$}\,\}
$$
denote the largest rate of uniform exponential convergence.

\medskip
We consider the absolute and normalized error criteria.
For $\varepsilon\in (0,1)$, $s\in \NN$, and
$\Lambda\in\{\Lambda^{\mathrm{all}},\Lambda^{\mathrm{std}}\}$,
the \emph{information complexity for the absolute error criterion} is defined as
\[
  n^{\bL_{\infty}\mathrm{-app},\Lambda}_{\mathrm{abs}}(\varepsilon,s):=
  \min\left\{n\,:\,e^{\bL_{\infty}\mathrm{-app},\Lambda}(n,s)
\le\varepsilon \right\}.
\]
Hence,
$n^{\bL_{\infty}\mathrm{-app},\Lambda}_{\mathrm{abs}}(\varepsilon,s)$
is the minimal
number of information evaluations from $\Lambda$
which is required  to achieve an error of at most $\e$.

For $\varepsilon\in (0,1)$, $s\in \NN$, and
$\Lambda\in\{\Lambda^{\mathrm{all}},\Lambda^{\mathrm{std}}\}$,
the \emph{information complexity for the normalized error criterion}
is defined as
\[
  n^{\bL_{\infty}\mathrm{-app},\Lambda}_{\mathrm{norm}}(\varepsilon,s):=
  \min\left\{n\,:\,e^{\bL_{\infty}\mathrm{-app},\Lambda}(n,s)
\le\varepsilon e^{\bL_{\infty}\mathrm{-app}}(0,s)\right\}.
\]
Thus,
$n^{\bL_{\infty}\mathrm{-app},\Lambda}_{\mathrm{norm}}(\varepsilon,s)$
is the minimal
number of information evaluations from $\Lambda$
which is required to reduce the initial error
$e^{\bL_{\infty}\mathrm{-app}}(0,s)$ by a
factor of $\varepsilon \in (0,1)$.

\medskip

In this paper, we study four different cases, namely
\begin{itemize}
 \item the absolute error criterion
       with information from $\Lambda^{\mathrm{std}}$,
 \item the absolute error criterion
       with information from $\Lambda^{\mathrm{all}}$,
 \item the normalized error criterion
       with information from $\Lambda^{\mathrm{std}}$,
 \item the normalized error criterion
       with information from $\Lambda^{\mathrm{all}}$.
\end{itemize}

There are several relations between these cases
which will be helpful in the analysis. First, note that clearly
\begin{equation}\label{errorstdall}
e^{\bL_{\infty}\mathrm{-app},\Lambda^{\mathrm{all}}}(n,s)\,
\le\,
e^{\bL_{\infty}\mathrm{-app},\Lambda^{\mathrm{std}}}(n,s),
\end{equation}
and therefore
\begin{equation}\label{complexitystdall}
n_{\mathrm{setting}}^{\bL_{\infty}\mathrm{-app},\Lambda^{\mathrm{all}}}(\e,s)
\,\le\,
n_{\mathrm{setting}}^{\bL_{\infty}\mathrm{-app},\Lambda^{\mathrm{std}}}(\e,s)
\end{equation}
where $\mathrm{setting} \in \{\mathrm{abs},\mathrm{norm}\}$.
Furthermore, since 
$e^{\bL_{\infty}\mathrm{-app}}(0,s)>1$ we have
\begin{equation}\label{complexityabsnorm}
 n^{\bL_{\infty}\mathrm{-app},\Lambda}_{\mathrm{norm}} 
\le n^{\bL_{\infty}\mathrm{-app},\Lambda}_{\mathrm{abs}}
\ \ \ \ \ \mbox{for $\Lambda\in
\{\Lambda^{\mathrm{all}},\Lambda^{\mathrm{std}}\}$}.
\end{equation}

\medskip

We are ready to define tractability concepts similarly
to \cite{DKPW13, DLPW11, KPW12, KPW14}, and we use
the name \textit{Exponential Convergence (EC) Tractability} for these
concepts, as introduced in~\cite{KPW14}. Following the recent paper of 
Petras and Papageorgiou~\cite{PP13},
we also study $\kappa$-EC-WT which is defined for $\kappa\ge1$.
We stress again that all these concepts correspond to the standard
concepts of tractability with  $\e^{-1}$ replaced by
$1+\log\,\e^{-1}$.

For $\Lambda\in\{\Lambda^{\mathrm{all}},\Lambda^{\mathrm{std}}\}$
and $\mathrm{setting} \in \{\mathrm{abs},\mathrm{norm}\}$, we
say that we have:
\begin{itemize}
\item \emph{$\kappa$-Exponential Convergence-Weak Tractability ($\kappa$-EC-WT)} for $\kappa\ge 1$ if 
$$
\lim_{s+\log\,\e^{-1}\to\infty}\frac{\log\
  n^{\bL_{\infty}\mathrm{-app},\Lambda}_{\mathrm{setting}}(\varepsilon,s)}{s+[\log\,\e^{-1}]^\kappa}=0.
$$
Here we set $\log\,0=0$ by convention. For $\kappa=1$ we say that we have \emph{Exponential Convergence-Weak Tractability (EC-WT)}.
\item \emph{Exponential Convergence-Polynomial Tractability (EC-PT)}
if there exist non-negative
  numbers $c,\tau_1$ and $\tau_2$ such that
$$
n^{\bL_{\infty}\mathrm{-app},\Lambda}_{\mathrm{setting}}(\varepsilon,s)\le
c\,s^{\,\tau_1}\,(1+\log\,\e^{-1})^{\,\tau_2}\ \ \ \ \ \mbox{for all}\ \ \ \
s\in\NN, \ \e\in(0,1).
$$
\item
\emph{Exponential Convergence-Strong Polynomial Tractability (EC-SPT)}
if there exist non-negative
  numbers $c$ and $\tau$  such that
$$
n^{\bL_{\infty}\mathrm{-app},\Lambda}_{\mathrm{setting}}(\varepsilon,s)\le
c\,(1+\log\,\e^{-1})^{\,\tau}\ \ \ \ \ \mbox{for all}\ \ \ \
s\in\NN, \ \e\in(0,1).
$$
The exponent $\tau^*$ of EC-SPT is defined as
the infimum of $\tau$ for which EC-SPT holds.
\end{itemize}
\vskip 1pc
Let us state some remarks about these definitions.

\medskip

Note that for $\kappa=1$ we obtain EC-WT, whereas for $\kappa>1$, the
notion of EC-WT is relaxed. The results for $\kappa=1$
and $\kappa>1$ can be quite different.

It is easy to see that if EC-PT holds for $\Lambda\in \{\Lambda^{\rm
  all},\Lambda^{\rm std}\}$ and for the absolute or normalized error 
criterion, then UEXP holds as well. Indeed,
due to \eqref{complexitystdall} and \eqref{complexityabsnorm}, 
it is sufficient to show this 
result for $\Lambda^{\rm all}$ and the normalized setting. 
Then EC-PT means that 
\[
 n^{\bL_{\infty}\mathrm{-app},
\Lambda^{\rm all}}_{\mathrm{norm}}(\varepsilon, s) \le 
 c s^{\tau_1}(1+\log\varepsilon^{-1})^{\tau_2},
\]
which implies 
\[
 e_{\mathrm{norm}}^{\bL_{\infty}
\mathrm{-app},\Lambda^{\rm all}}(n,s)\le 
\mathrm{e}^{1-((n/(cs^{\tau_1}))^{1/\tau_2}} 
 \prod_{j=1}^s \left(1+2\sum_{h=1}^\infty \omega^{a_j h^{b_j}}\right).
\]
Hence, we have UEXP, as claimed. 

For the \textit{absolute error criterion}, we
note, as in \cite{DKPW13, DLPW11}, that if \eqref{exrate} holds then
\begin{equation}\label{exrate2}
n^{\bL_{\infty}\mathrm{-app},\Lambda}_{\mathrm{abs}}(\e,s)
\le \left\lceil C_1(s) \left(\frac{\log C(s) +
    \log \e^{-1}}{\log q^{-1}}\right)^{1/p(s)}\right\rceil
\ \ \ \ \ \mbox{for all}\ \ \ s\in \NN\ \ \mbox{and}\ \ \e\in (0,1).
\end{equation}
Furthermore, if~\eqref{exrate2} holds then
$$
e^{\bL_{\infty}\mathrm{-app},\Lambda}(n+1,s)\le C(s)\, q^{\,(n/C_1(s))^{\,p(s)}}\ \
\ \ \ \mbox{for all}\ \ \ s,n\in \NN.
$$
This means that~\eqref{exrate} and~\eqref{exrate2} are practically
equivalent. Note that $1/p(s)$ determines the power of $\log\,\e^{-1}$
in the information complexity,
whereas $\log\,q^{-1}$ affects only the multiplier of $\log^{1/p(s)}\e^{-1}$.
{}From this point of view, $p(s)$ is more
important than $q$. That is why
we would like to have~\eqref{exrate} with the largest possible $p(s)$.
We shall see how to find such $p(s)$
for the parameters $(\bsa,\bsb)$ of the weighted Korobov space.

For the \textit{normalized error criterion}, we
replace $\e$ by $\e$ multiplied by the initial error.
If the initial
error is of order one for all $s$ we obtain the same results 
for both error criteria. On the other hand, if the initial error
is badly normalized this may change
tractability results. Note, however, that exponential convergence
is independent of the error criteria.

For both error criteria, exponential convergence
implies that asymptotically, with
respect to $\e$ tending to zero, we need $\mathcal{O}(\log^{1/p(s)}
\e^{-1})$ information evaluations to compute an
$\e$-approximation to functions from the Korobov space. However, it is not
clear how long we have to wait to see this nice asymptotic
behavior especially for large $s$. This, of course, depends on
how $C(s),C_1(s)$ and $p(s)$ depend on $s$. This is
the subject of tractability which is extensively studied
in many papers. So far tractability has been usually studied
in terms of $s$ and $\varepsilon^{-1}$. The current state of
the art on tractability can be found in~\cite{NW08,NW10,NW12}.
In this paper we follow the approach of \cite{DKPW13,DLPW11,KPW12,KPW14}
and we study tractability in terms of $s$ and $1+\log
\varepsilon^{-1}$.

\section{Main result}\label{mainresult}
In this section we present results for $\bL_\infty$-approximation.
The proofs of these results will be given in the subsequent sections.
\begin{theorem}\label{mainthm}
Consider $\bL_{\infty}$-approximation defined over the Korobov space with
kernel $K_{s,\bsa,\bsb}$ with arbitrary sequences $\bsa$ and
$\bsb$ satisfying~\eqref{coefficients}. 
The following results hold for 
$\Lambda\in\{\Lambda^{\rm{all}},\Lambda^{\rm{std}}\}$
and for the absolute and normalized error criterion.
\begin{enumerate}
\item
EXP holds for arbitrary $\bsa$ and $\bsb$ satisfying~\eqref{coefficients} and
$$
p^{*}(s)= 1/B(s) \ \ \ \ \ \mbox{with}\ \ \ \ \ B(s):=\sum_{j=1}^s\frac1{b_j}.
$$
\item
UEXP holds iff $\bsa$ is an arbitrary sequence
and $\bsb$ such that
$$
B:=\sum_{j=1}^\infty\frac1{b_j}<\infty.
$$
If so then $p^*=1/B$.
\item
$\kappa$-EC-WT for $\kappa\ge1$ holds \ iff \ $\lim_{j\to\infty}a_j=\infty$.
\item
$\mbox{EC-WT+UEXP}$ holds \ iff \ $B<\infty\ \ \mbox{and}\ \
\lim_{j\to\infty}a_j=\infty$. 
\item
The following notions are equivalent:
\begin{multline*}
\ \ \ \ \mbox{EC-PT}\ \Leftrightarrow\ \mbox{EC-PT+EXP}\ \Leftrightarrow\
\mbox{EC-PT+UEXP}\\ \Leftrightarrow\ \mbox{EC-SPT}\ \Leftrightarrow\
\mbox{EC-SPT+EXP}\ \Leftrightarrow\ \mbox{EC-SPT+UEXP}.
\end{multline*}
\item
EC-SPT+UEXP holds iff 
\begin{equation}\label{eqalphastar}
B:=\sum_{j=1}^\infty\frac1{b_j}<\infty\quad\mbox{and}\quad
\alpha^*:=\liminf_{j\to\infty}\frac{\log\,a_j}j>0.
\end{equation}
If so, then $\tau^*\in \left[B,B+\tfrac{\log 3}{\alpha^*}\right]$. 
In particular, if $\alpha^*=\infty$, then $\tau^*=B$.
\end{enumerate}
\qed
\end{theorem}
We now briefly comment on Theorem \ref{mainthm}.
We find it surprising that the results are the same for 
$\Lambda^{\rm all}$ and $\Lambda^{\rm std}$ 
and they do not depend on the error criteria.

Exponential convergence holds for all $\bsa$ and~$\bsb$ 
satisfying~\eqref{coefficients}. What is more, the rate $p^*(s)$ 
is independent of $\bsa$ and depends only on $\bsb$. Note that
$B(s)\le s/b_*$ and therefore $p^*(s)\ge b_*/s$, and the last bound
is sharp if $b_j=b_*$ for all $j\in\NN$. 
In this case $p^*(s)$ is small for large $s$ and tends to zero as $s$
approaches infinity. On the other hand, uniform exponential convergence 
holds independently of $\bsa$ and only for summable $b_j^{-1}$.
Obviously, $B$ can be arbitrarily large and $p^*$
arbitrarily small. 

The notion of $\kappa$-EC-WT for $\kappa\ge 1$ 
is independent of $\bsb$ and holds iff $a_j$ goes to infinity. 
We stress that the rate how fast $a_j$ goes to infinity is irrelevant. 
We shall see later that for $\bL_2$-approximation 
the result is different since for $\kappa>1$ and
the class $\Lambda^{\rm all}$, the notion of $\kappa$-EC-WT 
holds for all $\bsa$ and $\bsb$. 
 
The notion of EC-WT does not necessarily imply uniform exponential 
convergence since EC-WT holds for all $\bsb$. To guarantee EC-WT and
UEXP we must assume summable $b_j^{-1}$ and $a_j$ converging to
infinity.

The next point of Theorem \ref{mainthm} shows that a number of
tractability notions are equivalent for $\bL_\infty$-approximation. 
Probably, the most interesting 
one is that EC-PT is equivalent to EC-SPT+UEXP. In particular, there is
no difference between EC-PT and EC-SPT. 

Based on these equivalences, it is therefore enough to find necessary and
sufficient conditions for EC-SPT+UEXP. It turns out that this holds iff 
$b_j^{-1}$'s are summable and $a_j$'s are
exponentially large in $j$.

\section{Relations to $\bL_2$-approximation}\label{relL2}

In \cite{DKPW13} we studied $\bL_2$-approximation of
functions from $H(K_{s,\bsa,\bsb})$.
This problem 
$$
{\rm EMB}_{s,2}:H(K_{s,\bsa,\bsb}) 
\rightarrow \bL_2([0,1]^s)\ \ \ \ \
\mbox{given by}\ \ \ \ \ {\rm EMB}_{s,2}(f)=f
$$
is defined as an approximation of the embedding
from  the Korobov space $H(K_{s,\bsa,\bsb})$ to the space
$\bL_2([0,1]^s)$.
Again for this problem it is enough to use
linear algorithms~$A_{n,s}$ of the form
\begin{equation*}
  A_{n,s}(f) = \sum_{k=1}^{n}\alpha_k L_k(f)\ \ \ \
\mbox{for} \ \ \ \ \ f \in H(K_{s,\bsa,\bsb}),
\end{equation*}
where each $\alpha_k$ is a function from $\bL_{2}([0,1]^{s})$ and
each $L_k$ is a continuous linear functional defined on
$H(K_{s,\bsa,\bsb})$ from the class
$\Lambda \in \{\Lambda^{{\rm all}},\Lambda^{{\rm std}}\}$.

In the same vein as for the $\bL_{\infty}$-case the
\textit{worst-case error} of the algorithm $A_{n,s}$ is
now defined as
\[
  e^{\bL_2\mathrm{-app}}(H(K_{s,\bsa,\bsb}),A_{n,s})
  :=  \sup_{f \in H(K_{s,\bsa,\bsb}) \atop \norm{f}_{H(K_{s,\bsa,\bsb})}\le 1}
  \norm{f-A_{n,s}(f)}_{\bL_{2}([0,1]^s)},
\]
and the $n$th minimal worst-case error is defined by
$$
e^{\bL_2\mathrm{-app},\Lambda}(n,s) = \inf_{A_{n,s}}
e^{\bL_2\mathrm{-app}}(H(K_{s,\bsa,\bsb}),A_{n,s}),
$$
where the infimum is taken
over all linear algorithms $A_{n,s}$ using $n$ information evaluations
from the class $\Lambda$. For $n=0$ we obtain the initial error
$$
e^{\bL_2\mathrm{-app},\Lambda}(0,s) = 1,
$$ 
as shown in \cite{DKPW13}. Hence, there is no difference between 
the absolute and normalized error criteria for $\bL_2$-approximation.

For $\varepsilon\in (0,1)$, $s\in \NN$, and
$\Lambda\in\{\Lambda^{\mathrm{all}},\Lambda^{\mathrm{std}}\}$,
the \emph{information complexity} 
(for both the absolute and normalized error criteria)
is defined as
\[
  n^{\bL_2\mathrm{-app},\Lambda}(\varepsilon,s):=
  \min\left\{n\,:\,e^{\bL_2\mathrm{-app},\Lambda}(n,s)\le\varepsilon \right\}.
\]

It is easy to show that $\bL_2$-approximation is not harder than
$\bL_{\infty}$-approximation for the absolute error criterion. Namely we
have the following lemma.

\begin{lemma}\label{l2linfty}
For $\Lambda \in \{\Lambda^{{\rm all}},\Lambda^{{\rm std}}\}$ we have
\begin{equation}\label{errorl2linfty}
e^{\bL_2\mathrm{-app},\Lambda}(n,s) 
\le e^{\bL_{\infty}\mathrm{-app},\Lambda}(n,s)
\end{equation}
and therefore
\begin{equation}\label{complexityl2linfty}
n^{\bL_2\mathrm{-app},\Lambda}(\varepsilon,s) \le
n^{\bL_{\infty}\mathrm{-app},\Lambda}_{\mathrm{abs}}(\varepsilon,s).
\end{equation}
\qed
\end{lemma}

\begin{proof}
Note that any algorithm $A_{n,s}=\sum_{k=1}^n \alpha_k L_k$
with $\alpha_k \in \bL_\infty=\bL_{\infty}([0,1]^s)$ is
also an algorithm $A_{n,s}=\sum_{k=1}^n \alpha_k L_k$
with $\alpha_k \in \bL_2=\bL_2([0,1]^s)$. Thus, the class of
admissible linear algorithms for $\bL_{\infty}$-approximation is
contained in the class of admissible linear algorithms
for $\bL_{2}$-approximation.
Furthermore,
\begin{eqnarray*}
e^{\bL_2\mathrm{-app},\Lambda}(n,s) & = & \inf_{A_{n,s} \atop
\alpha_k \in \bL_2} \sup_{f \in H(K_{s,\bsa,\bsb}) \atop
\norm{f}_{H(K_{s,\bsa,\bsb})}\le 1}
  \norm{f-A_{n,s}(f)}_{\bL_{2}([0,1]^s)}\\
& \le & \inf_{A_{n,s} \atop \alpha_k \in \bL_\infty}
\sup_{f \in H(K_{s,\bsa,\bsb}) \atop \norm{f}_{H(K_{s,\bsa,\bsb})}\le 1}
\norm{f-A_{n,s}(f)}_{\bL_{2}([0,1]^s)}\\
& \le & \inf_{A_{n,s} \atop \alpha_k \in \bL_{\infty}}
\sup_{f \in H(K_{s,\bsa,\bsb}) \atop \norm{f}_{H(K_{s,\bsa,\bsb})}\le 1}
\norm{f-A_{n,s}(f)}_{\bL_{\infty}([0,1]^s)}\\
& = & e^{\bL_{\infty}\mathrm{-app},\Lambda}(n,s),
\end{eqnarray*}
which proves \eqref{errorl2linfty} and implies~\eqref{complexityl2linfty}.
\end{proof}

The notions of (U)EXP, $\kappa$-EC-WT, EC-WT, EC-PT, EC-SPT 
for $\bL_2$-approximation in $H(K_{s,\bsa,\bsb})$ 
are defined in the same way as for $\bL_\infty$-approximation 
in $H(K_{s,\bsa,\bsb})$ but with 
$e^{\bL_{\infty}\mathrm{-app},\Lambda}(n,s)$ 
replaced by $e^{\bL_2\mathrm{-app},\Lambda}(n,s)$ and $
n^{\bL_{\infty}\mathrm{-app},\Lambda}_{\mathrm{setting}}
(\varepsilon,s)$ by $n^{\bL_2\mathrm{-app},\Lambda}(\varepsilon,s)$.

We will be using the results for $\bL_2$-approximation proved
in~\cite{DKPW13}.

\begin{theorem}[{\cite[Theorem~1]{DKPW13}}]\label{thm1in2}
Consider $\bL_2$-approximation defined over the Korobov space
$H(K_{s,\bsa,\bsb})$ with weight sequences $\bsa$ and $\bsb$
satisfying~\eqref{coefficients}. The following results hold for both
classes $\Lambda^{\rm all}$ and $\Lambda^{\rm std}$.
\begin{itemize}
\item
EXP holds for all considered $\bsa$ and $\bsb$ with
$p^*(s)=1/B(s)$, where $B(s)=\sum_{j=1}^sb_j^{-1}$.
\item UEXP holds iff $\bsa$ is an arbitrary sequence and $\bsb$ is
  such that $B=\sum_{j=1}^\infty b_j^{-1}<\infty$. If so then $p^*=1/B$.
\item
$\mbox{EC-WT}$ holds iff $\lim_{j\to\infty}a_j=\infty$.
\item
The notions of $\mbox{EC-PT}$ and $\mbox{EC-SPT}$ are equivalent, and
hold iff
$$B=\sum_{j=1}^\infty b_j^{-1}<\infty\ \ \ \mbox{and}\ \ \
\alpha^*=\liminf_{j\to\infty}\frac{\log\, a_j}{j}>0.
$$
If so then
$
\tau^*\in\left[B,B+\min\left(B,\frac{\log\,3}{\alpha^*}\right)\right]$.
In particular, if $\alpha^*=\infty$ then $\tau^*=B$.
\end{itemize}
\qed
\end{theorem}
\vskip 1pc
For the class $\Lambda^{\rm all}$ we have the full characterization of
$\bL_2$- and $\bL_{\infty}$-approximation in terms of the eigenpairs of the
operator $W_s={\rm EMB}_{s,2}^*{\rm EMB}_{s,2}:\ 
H(K_{s,\bsa,\bsb})\to H(K_{s,\bsa,\bsb})$,
which is given by
$$
W_sf=\int_{[0,1]^s} f(\bst) K_{s,\bsa,\bsb} (\cdot,\bst)\rd\bst.
$$
For $\bL_2$-approximation this result is standard and may be found
for instance in \cite{NW08} and \cite{TWW88}, 
whereas for $\bL_{\infty}$-approximation it was proved in 
\cite[Theorem~4 in Section~3]{KWW08} (with $\rho\equiv 1$).

More precisely, for $\bsh\in\ZZ^s$, let a function $e_{\bsh}$ be defined by
$$
e_{\bsh}(\bsx)=\exp(2\pi\icomp\,\bsh\cdot\bsx)\,
\omega_{\bsh}^{1/2}\ \ \ \ \
\mbox{for all}\ \ \ \ \ \bsx \in[0,1]^s.
$$
Then 
$\{e_{\bsh}\}_{\bsh\in\ZZ^s}$ is a complete
orthonormal basis of the Korobov space $H(K_{s,\bsa,\bsb})$.
It is easily checked that the eigenpairs of
$W_s$ are $(\omega_{\bsh},e_{\bsh})$, i.e.,
$$
W_se_{\bsh}=\omega_{\bsh}e_{\bsh}=
\omega^{\,\sum_{j=1}^sa_j|h_j|^{b_j}}\, e_{\bsh}\ \ \ \ \
\mbox{for all}\ \ \ \ \ \bsh\in \ZZ^s,
$$
see also \cite[Section~5]{DKPW13}.
Let the ordered eigenvalues of
$W_s$ be $\{\lambda_{s,k}\}_{k\in\nat}$ with
$$
\lambda_{s,1} \ge \lambda_{s,2}\ge \lambda_{s,3}\ge \ldots.
$$
Obviously,
$\{\lambda_{s,k}\}_{k\in\nat}=\{\omega_{\bsh}\}_{\bsh\in\ZZ^s}$
and $\lambda_{s,1}=1$. Then
\begin{eqnarray*}
e^{\bL_2-\mathrm{app},\Lambda^{\rm all}}(n,s)&=&\lambda_{s,n+1}^{1/2}\\
e^{\bL_\infty-\mathrm{app},\Lambda^{\rm all}}(n,s)&=&
\left(\sum_{k=n+1}^\infty\lambda_{s,k}\right)^{1/2}.
\end{eqnarray*}
Let
$$
{\rm CRI}_{\mathrm{abs}}=1\ \ \ \mbox{and}\ \ \
{\rm CRI}_{\mathrm{norm}}=
e^{\bL_{\infty}\mathrm{-app},\Lambda^{\rm all}}(0,s) =
\prod_{j=1}^s \left(1+2 \sum_{h=1}^{\infty} \omega^{a_j
    h^{b_j}}\right)^{1/2}.
$$
Then for $\mathrm{setting}\in\{\mathrm{abs},\mathrm{norm}\}$ we have
\begin{eqnarray}
n^{\bL_2\mathrm{-app},\Lambda^{\rm all}}(\e,s)
&=&
 \min\bigg\{n \ : \ \lambda_{s,n+1} \le
 \varepsilon^2\,\bigg\},\\
n^{\bL_{\infty}\mathrm{-app},
\Lambda^{\rm all}}_{\mathrm{setting}}(\e,s)\label{imp1}
&=&
 \min\bigg\{n \ : \
\sum_{k=n+1}^\infty \lambda_{s,k} \le \varepsilon^2\,
\mathrm{CRI}^2_{\mathrm{setting}}\bigg\}\label{imp2}.
\end{eqnarray}
Furthermore the $n$th minimal errors are attained
for both $\bL_2$- and $\bL_{\infty}$-approximation by
the same algorithm
\begin{equation}\label{algoptimal}
A_{n,s}(f)=\sum_{k=1}^n \langle f, \eta_{s,k}
\rangle_{H(K_{s,\bsa,\bsb})}\ \eta_{s,k},
\end{equation}
where the $\eta_{s,k}$'s are the eigenfunctions $e_{\bsh}$ corresponding to
the ordered eigenvalues $\lambda_{s,k}$. That is,
$\eta_{s,k}=e_{\bsh(k)}$ and $\lambda_{s,k}=\omega_{\bsh(k)}$ for some
$\bsh(k)\in\ZZ^s$. Note that for any $f\in H(K_{s,\bsa,\bsb})$ we have
\begin{eqnarray*}
\il f,\eta_{s,k}\ir_{L_2([0,1]^s)}&=&
\il {\rm EMB}_{s,2}f,{\rm EMB}_{s,2}\, \eta_{s,k}\ir_{L_2([0,1]^s)}\\
&=&
\il f,W_s\eta_{s,k}\ir_{H(K_{s,\bsa,\bsb})}=
\lambda_{s,k}\,\il f, \eta_{s,k}\ir_{H(K_{s,\bsa,\bsb})}.
\end{eqnarray*}
Therefore, \eqref{algoptimal} can be equivalently rewritten as
\begin{equation}\label{algoptimal2}
A_{n,s}(f)=\sum_{k=1}^n\il f,\eta_{s,k}\ir_{L_2([0,1]^s)}
\lambda_{s,k}^{-1}\, \eta_{s,k}=
\sum_{k=1}^n 
\il f,\widetilde{e}_k\ir_{L_2([0,1]^s)}\, \widetilde{e}_k
\end{equation}
where
$$
\widetilde{e}_k(\bsx)=\frac{\eta_{s,k}(\bsx)}{\sqrt{\lambda_{s,k}}}=
\frac{e_{\bsh(k)}(\bsx)}
{\sqrt{\omega_{\bsh(k)}}}=\exp(2\pi\,\icomp\,\bsh(k)\cdot\bsx).
$$
Clearly, $\widetilde{e}_k$'s are orthonormal in $\bL_2([0,1]^s$ and 
$\|\widetilde{e}_k\|_{\bL_\infty([0,1]^s)}=1$ for all $k\in\NN$. 

We now find an estimate on the $n$th minimal error for
$\bL_\infty$-approximation and the class $\Lambda^{\rm std}$
in terms of the $n$th minimal errors for $\bL_\infty$-approximation 
and the class $\Lambda^{\rm all}$, and for $\bL_2$-approximation
and  the class $\Lambda^{\rm std}$. 
\begin{lemma}\label{lemma2}
We have 
$$
e^{\bL_\infty-\mathrm{app},\Lambda^{\rm std}}(n,s)\ \le\ 
e^{\bL_\infty-\mathrm{app},\Lambda^{\rm all}}(n,s)\ +\ 
n\,e^{\bL_2-\mathrm{app},\Lambda^{\rm std}}(n,s).
$$
\qed
\end{lemma}
\begin{proof}
Consider a linear algorithm $B_{n,s}$ that uses $n$ function values
for $\bL_2$-approximation,
$$
B_{n,s}(f)=\sum_{j=1}^n\alpha_j\,f(\bsx_j)\ \ \ \ \ 
\mbox{for}\ \ f\in H(K_{s,\bsa,\bsb}),
$$
where $\alpha_j\in \bL_2([0,1]^s)$ and $\bsx_j\in [0,1]^s$.

We now approximate  the algorithm $A_{n,s}$ given by
\eqref{algoptimal2} by replacing $f$ in the inner product of
$\bL_2([0,1]^s)$ by $B_{n,s}(f)$,
$$
\widetilde{A}_{n,s}(f)=\sum_{k=1}^n\il
B_{n,s}(f),\widetilde{e}_k\ir_{\bL_2([0,1]^s)}\,\widetilde{e}_k=
\sum_{j=1}^nf(\bsx_j)\left(
\sum_{k=1}^n  
\il \alpha_j,\widetilde{e}_k\ir_{\bL_2([0,1]^s)}\,\widetilde{e}_k\right).
$$
This means that the algorithm $\widetilde{A}_{n,s}$ uses at most $n$
function values. Furthermore,
$$
A_{n,s}(f)-\widetilde{A}_{n,s}(f)=\sum_{k=1}^n\il
f-B_{n,s}(f),\widetilde{e}_k\ir_{\bL_2([0,1]^s)}\,\widetilde{e}_k,
$$
which implies
\begin{eqnarray*}
\|A_{n,s}(f)-\widetilde{A}_{n,s}(f)\|_{\bL_\infty([0,1]^s)}&\le&
n\,\|f-B_{n,s}(f)\|_{\bL_2([0,1]^s)}\\
&\le&
n\,\|f\|_{H(K_{s,\bsa,\bsb})}\,
e^{\bL_2\mathrm{-app}}(H(K_{s,\bsa,\bsb}),B_{n,s}).
\end{eqnarray*}
Hence,
\begin{eqnarray*}
\|f-\widetilde{A}_{n,s}\|_{\bL_\infty([0,1]^s)}&\le&
\|f-A_{n,s}(f)\|_{\bL_\infty([0,1]^s)}+
\|A_{n,s}(f)-\widetilde{A}_{n,s}(f)\|_{\bL_\infty([0,1]^s)}\\
&\le&\|f\|_{H(K_{s,\bsa,\bsb})}\left(
e^{\bL_\infty-\mathrm{app},\Lambda^{\rm all}}(n,s)\,+\,
n\,e^{\bL_2\mathrm{-app}}(H(K_{s,\bsa,\bsb}),B_{n,s})\right).
\end{eqnarray*}
Choosing $B_{n,s}$ as an optimal algorithm for $\bL_2$-approximation
and the class $\Lambda^{\rm std}$ we obtain
$$
e^{\bL_\infty-\mathrm{app},\Lambda^{\rm std}}(n,s)\,\le\,
e^{\bL_\infty-\mathrm{app},\Lambda^{\rm all}}(n,s)\,+\,
n\,e^{\bL_2-\mathrm{app},\Lambda^{\rm std}}(n,s),
$$
as claimed. 
\end{proof}
\vskip 1pc
Lemma \ref{lemma2} and known estimates on $\bL_2$-approximation
for the class $\Lambda^{\rm std}$ allow us to find an estimate 
on the $n$th minimal error of $\bL_\infty$-approximation for the class
$\Lambda \in \{\Lambda^{\rm all},\Lambda^{\rm std}\}$  
in terms of the eigenvalues $\lambda_{s,n}$.
 \begin{lemma}\label{lemma3}
Assume that for all $s\in\NN$ there are positive numbers $\beta_s$ and
$M_s>0$ such that
$$
\lambda_{s,n}\le \frac{M_s^2}{n^{2\,\beta_s}}\ \ \ \ \ \mbox{for all}
\ \ n\in \NN.
$$
We assume for the class $\Lambda^{\rm all}$ that $\beta_s>\tfrac12$,
and for the class $\Lambda^{\rm std}$ that $\beta_s>\tfrac32$. Then 
\begin{eqnarray*}
e^{\bL_\infty-\mathrm{app},\Lambda^{\rm all}}(n,s)&\le&
  \frac{M_s}{\sqrt{2\,\beta_s-1}}\ \frac 1{n^{\beta_s-1/2}},\\
e^{\bL_\infty-\mathrm{app},\Lambda^{\rm std}}(n,s)&\le&
M_s\left(\frac{\sqrt{2}}2+C(\beta_s)\right)\,\frac1{n^{\beta_s-3/2}},
\end{eqnarray*}
where $C(x)=2^{2x(2x+1)+x-1/2}((2x+1)/(2x-1))^{1/2}(1+1/(2x))^x$.
\qed
\end{lemma}

\begin{proof}
For the class $\Lambda^{\rm all}$, we easily have
\begin{eqnarray*}
\left[e^{\bL_\infty-\mathrm{app},\Lambda^{\rm all}}(n,s)\right]^2
&=& \sum_{k=n+1}^\infty\lambda_{s,k}\le
M_s^2\,\sum_{k=n+1}^\infty\frac1{k^{2\,\beta_s}}\\
&\le& M_s^2\,\int_n^\infty\frac{{\rm
    d}x}{x^{2\,\beta_s}}=\frac{M_s^2}{2\,\beta_s-1}\,\frac1{n^{2\,\beta_s-1}},
 \end{eqnarray*}
as claimed. 

For the class $\Lambda^{\rm std}$, we use \cite[Theorem 26.15]{NW12} 
which states that
$$
e^{\bL_2-\mathrm{app},\Lambda^{\rm std}}(n,s)\le
\frac{M_sC(\beta_s)}{n^{\beta_s-1/2}}.
$$
{}From Lemma \ref{lemma2} we then have 
$$
e^{\bL_\infty-\mathrm{app},\Lambda^{\rm std}}(n,s)\le
\frac{M_s}{\sqrt{2\,\beta_s-1}}\,\frac 1{n^{\beta_s-1/2}}
\,+\, \frac{M_sC(\beta_s)}{n^{\beta_s-3/2}}\le 
M_s\left(\frac{\sqrt{2}}{2}+C(\beta_s)\right)\,\frac
  1{n^{\beta_s-3/2}},
$$
as claimed. 
\end{proof}

\section{Preliminaries for $\Lambda^{\rm{std}}$}\label{secprelA}

Before we proceed to prove our main results, 
we state some preliminary observations that we need 
for $\bL_{\infty}$-approximation 
using the information class $\Lambda^{\rm{std}}$. 

We follow \cite{ZKH09} in our arguments 
and present a particular 
choice of a linear approximation algorithm based on 
function evaluations that allows us to obtain
error bounds.  Given a set of points
$\cP=\{\bsx_1,\ldots,\bsx_n\}$, and function evaluations,
$\{f(\bsx_1),\ldots,f(\bsx_{n})\}$, we define a spline $\sigma$ as
\[\sigma(f;\cP):=\mathrm{argmin}\{\norm{g}_{H(K_{s,\bsa,\bsb})}\,:\ 
g\in H(K_{s,\bsa,\bsb}),
g(\bsx_k)=f(\bsx_k), k=1,2,\ldots,n\}.\]

We would like to use $\sigma(f;\cP)$ 
for approximating $f\in
H(K_{s,\bsa,\bsb})$ in the $\bL_{\infty}$ norm. The first part of the
analysis in \cite{ZKH09} holds for reproducing kernels $K$ of
Hilbert spaces of 1-periodic functions, and it is required that
the kernels $K$ take the form
\[K(\bsx,\bsy)=\widetilde{K}(\{\bsx-\bsy\}),\]
where $\{\cdot\}$ denotes the fractional part of a real number
(defined component-wise). These assumptions are fulfilled
for the kernels $K_{s,\bsa,\bsb}$ considered here. Therefore the
preliminaries outlined in \cite{ZKH09} apply to our case as well,
and we restrict ourselves to summarizing the most crucial facts
from \cite{ZKH09}. In fact, the paper \cite{ZKH09}
discusses approximation algorithms that  use lattice points,
but the theory also applies to the case where we 
consider approximation by
$s$-dimensional grids $\cG_{n,s}$ as in this paper. 
Such regular grids have already been studied 
in~\cite{DKPW13, DLPW11, KPW12, KPW14}. We now
recall their definition. 

For $s \in \NN$, a regular grid with
mesh-sizes $m_1,\ldots,m_s
\in \NN$ is defined as the point set
\begin{equation*}
\cG_{n,s}
=\left\{(k_1/m_1,\ldots, k_s/m_s)\, : \ \
k_j=0,1,\ldots,m_j -1 \mbox{\ \  for all\  } j=1,2,\ldots ,s\right\},
\end{equation*}
where $n=\prod_{j=1}^sm_j$ is the cardinality of $\cG_{n,s}$.
By $\cG_{n,s}^{\bot}$ we denote the
dual of $\cG_{n,s}$, i.e.,
$$
\cG_{n,s}^{\bot}=
\{\bsl =(l_1,\ldots,l_s) \in \ZZ^s \, : \, l_j \equiv 0 \pmod{m_j}
\ \mbox{ for all } \ j=1,2,\ldots,s\}.
$$

For  $\cG_{n,s}$ with mesh-sizes
$m_1,\ldots,m_s\in \NN$, and cardinality $n=m_1\cdots m_s$, we write the set
$\ZZ^s$ as a direct sum of $\cG_{n,s}^\bot$ and the set
\begin{equation}\label{eqVn}
\calV_n= \ZZ^s \cap \prod_{j=1}^s
\left(-\frac{m_j}{2},\frac{m_j}{2}\right],
\end{equation}
i.e.,
\[\ZZ^s= \calV_n \oplus \cG_{n,s}^\bot=\{\bsv+\bsl\, : \, \bsv \in \calV_n \mbox{ and } \bsl \in \cG_{n,s}^\bot\}.\]
Note that $\calV_n$ has the property that any
two distinct vectors in $\calV_n$ differ by a vector that is not
in the dual set $\cG_{n,s}^\bot$, i.e.,
\[\bsv,\bsw\in \calV_n, \bsv\neq\bsw\ \Rightarrow\ \bsv-\bsw\notin
\cG_{n,s}^\bot\setminus\{\bszero\}.\] Furthermore, $\bszero\in \calV_n$ and 
\begin{equation}\label{monomega}
\omega_{\bsv}^{-1}\le
\omega_{\bsv+\bsl}^{-1} \ \ \ \ \mbox{ for all }\\ \bsv\in \calV_n \mbox{ and all }
\bsl\in \cG_{n,s}^\bot.
\end{equation}
This follows from the fact that for 
$\bsv=(v_1,\ldots,v_s)\in \calV_n$ and for 
$\bsl=(l_1,\ldots,l_s)\in \cG_{n,s}^\bot$
we have $|v_j| \le |v_j+l_j|$ for all $j=1,\ldots,s$.

Given $\cG_{n,s}$ with points $\bsx_1,\ldots,\bsx_{n}$, 
it is known, see \cite{ZKH09} and the references therein,
that the spline $\sigma(f;\cP)$ can be expressed
in terms of so-called cardinal functions, $\phi_k$,
$k=1,2,\ldots,n$, 
where each $\phi_k$ is a linear combination of the
$K_{s,\bsa,\bsb}(\cdot,\bsx_r)$. To be more precise,
\[\sigma(f;\cG_{n,s})(\bsx)=\sum_{k=1}^{n} f(\bsx_k)\phi_k (\bsx),\]
\[\phi_k(\bsx)=\sum_{r=1}^{n} K_{s,\bsa,\bsb} (\bsx,\bsx_r) \xi_{r,k},\]
where the $\xi_{r,k}$ are given by a condition expressed by the
Kronecker delta function $\delta$,
\[\delta_{j,k}=\sum_{r=1}^{n} K_{s,\bsa,\bsb} (\bsx_j,\bsx_r) \xi_{r,k}.\]

Going through analogous steps as in \cite[Section~3.1]{ZKH09}, we
arrive at an estimate similar to one formulated for lattice points in
\cite[Theorem~1]{ZKH09},
\begin{equation}\label{eqdoublesum}
\left[e^{\bL_{\infty}{\rm -app},\Lambda^{{\rm std}}}
(H(K_{s,\bsa,\bsb}),\cG_{n,s})\right]^2 
\le 4\sum_{\bsh\notin \calV_n}\omega_{\bsh}=
  4\sum_{\bsv\in \calV_n}
\sum_{\bsl\in \cG_{n,s}^\bot\setminus\{\bszero\}}\omega_{\bsv+\bsl},
\end{equation}
where here and in the following we just write 
$e^{\bL_{\infty}{\rm -app},\Lambda^{{\rm std}}}
(H(K_{s,\bsa,\bsb}),\cG_{n,s})$ instead of 
$e^{\bL_{\infty}{\rm -app},\Lambda^{{\rm std}}}
(H(K_{s,\bsa,\bsb}),\sigma(\cdot;\cG_{n,s}))$. 

It is easy to see that
\[\abs{l}^b\le 2^b\left(\abs{v+l}^b + \abs{v}^b\right)\]
for any $v,l\in\ZZ$ and any $b>0$.
{}From \eqref{monomega} we get for all $\bsv\in
\calV_n$ and all $\bsl\in \cG_{n,s}^\bot$ that
\begin{eqnarray*}
\omega_{\bsv +\bsl}& = & \omega^{\sum_{j=1}^s a_j \abs{v_j +
l_j}^{b_j}}\le \omega^{\sum_{j=1}^s 2^{-b_j} a_j\abs{l_j}^{b_j}}
  \omega^{-\sum_{j=1}^s a_j\abs{v_j}^{b_j}}\\
& = & \omega^{\sum_{j=1}^s 2^{-b_j}
a_j\abs{l_j}^{b_j}}\omega_{\bsv}^{-1}\le \omega^{\sum_{j=1}^s
2^{-b_j} a_j\abs{l_j}^{b_j}}\omega_{\bsv+\bsl}^{-1}.
\end{eqnarray*}
This implies
\[\omega_{\bsv +\bsl}\le (\omega^{1/2})^{\sum_{j=1}^s
2^{-b_j}a_j\abs{l_j}^{b_j}}.\]
Inserting this estimate into \eqref{eqdoublesum} we arrive at
\begin{eqnarray}\label{linftyapproxerr}
\left[e^{\bL_{\infty}{\rm -app},
\Lambda^{{\rm std}}}(H(K_{s,\bsa,\bsb}),\cG_{n,s})\right]^2 \le
 4\sum_{\bsv\in \calV_n}\sum_{\bsl\in
\cG_{n,s}^\bot\setminus\{\bszero\}} 
(\omega^{1/2})^{\sum_{j=1}^s 2^{-b_j}a_j \abs{l_j}^{b_j}}
 =  4n F_n,
\end{eqnarray}
where
\begin{equation}\label{eqdefFn}
F_n=\sum_{\bsl\in \cG_{n,s}^\bot\setminus\{\bszero\}}
\overline{\omega}^{\,\sum_{j=1}^s 2^{-b_j}a_j\abs{l_j}^{b_j}}
=-1+\prod_{j=1}^s
\left(1+2\sum_{h=1}^\infty
\overline{\omega}^{\,a_j 2^{-b_j}(m_j h)^{b_j}}\right),
\end{equation}
and where we write $\overline{\omega}:=\omega^{1/2}$.

\section{(Uniform) exponential convergence}\label{secexp}

In this section, we prove Points 1 and 2 of Theorem \ref{mainthm}
for EXP and UEXP. 

Let us first consider the result for the class $\Lambda^{\rm std}$. 
We now show how to choose a regular grid in the sense 
of Section \ref{secprelA} to obtain the desired result. 

Let $\omega_1\in (\overline{\omega},1)$. For $s\in \NN$ and $\e\in(0,1)$ define
$$
m=\max_{j=1,2,\dots,s}\
\left\lceil \left(
\frac{4^{b_j}}{a_j}\,\frac{\log\left(1+\frac{RC_j 2s}{\log(1+\e^2/4)}\right)}
{\log\,\omega_1^{-1}}\right)^{B(s)}\,\right\rceil,
$$
where
$$
C_j=\sup_{m\in\NN} \,m^{1/s} 
(\overline{\omega}/\omega_1)^{m^{1/B(s)} a_{\ast} 4^{-b_j}}<\infty,
$$
and
$$
R=\max_{1\le j\le s} \, 
\sum_{h=1}^\infty \omega_1^{a_j 4^{-b_j}(h^{b_j}-1)}<\infty.
$$

Let
$\cG_{n,s}^{\ast}$ be a regular grid with mesh-sizes $m_1,m_2,\ldots,m_s$
given by
$$
m_j:=\left\lfloor m^{1/(B(s) \cdot b_j)}\right\rfloor\ \ \ \ \
\mbox{for}\ \ \  j=1,2,\ldots,s\ \ \
\mbox{and}\ \ \ n=\prod_{j=1}^sm_j.
$$
We are now going to show that
\begin{equation}\label{eqgoal}
e^{\bL_{\infty}{\rm -app},
\Lambda^{{\rm std}}}(H(K_{s,\bsa,\bsb}),\cG_{n,s}^\ast)\le\e,\ \ \ \
\mbox{and}\ \ \ \
n=\mathcal{O}\left(\log^{\,B(s)}\left(1+\e^{-1}\right)\right)
\end{equation}
with the factor in the $\mathcal{O}$ notation independent
of $\e^{-1}$ but dependent on $s$.

{}From \eqref{eqdefFn} we have
\[
F_n=
-1+\prod_{j=1}^s\left(1+2\sum_{h=1}^\infty
\overline{\omega}^{\,a_j 2^{-b_j}(m_j h)^{b_j}}\right).
\]
Since $\lfloor x\rfloor\ge x/2$ for all $x\ge1$, we have
$$
(m_jh)^{b_j}\ge (h/2)^{b_j}\,m^{1/B(s)}
\qquad\mbox{for all}\qquad j=1,2,\dots,s.
$$
Hence,
$$
F_n \le -1+\prod_{j=1}^s
\left(1+2 \sum_{h=1}^{\infty} 
\overline{\omega}^{\,m^{1/B(s)} a_j4^{-b_j}\,h^{b_j}}\right),
$$
and similarly
$$
n F_n \le -1+\prod_{j=1}^s
\left(1+2\,n^{1/s}\,\sum_{h=1}^{\infty}
\overline{\omega}^{\,m^{1/B(s)} a_j4^{-b_j}\,h^{b_j}}\right).
$$
Note that
$$
n=\prod_{j=1}^s
m_j=\prod_{j=1}^s
\left\lfloor
m^{1/(B(s) \cdot b_j)} \right\rfloor \le m^{\frac{1}{B(s)}\sum_{j=1}^s
1/b_j} = m.
$$
Now with $q:=\overline{\omega}/\omega_1$ 
we have $q\in (0,1)$, and hence, for $h\ge 1$,
$$
n^{1/s}\,q^{m^{1/B(s)} a_j4^{-b_j}\,h^{b_j}}\le 
m^{1/s}\,q^{m^{1/B(s)} a_{\ast} 4^{-b_j}}\le 
\sup_{m\in\NN} \,m^{1/s}\,q^{m^{1/B(s)} a_{\ast} 4^{-b_j}}=C_j.
$$
Therefore,
$$
n F_n \le -1+\prod_{j=1}^s
\left(1+2 C_j \sum_{h=1}^{\infty}  
\omega_1^{m^{1/B(s)} a_j4^{-b_j}\,h^{b_j}}\right).
$$
We further estimate
\begin{eqnarray}
\sum_{h=1}^{\infty} \omega_1^{m^{1/B(s)} a_j4^{-b_j}\,h^{b_j}}&=&
\omega_1^{m^{1/B(s)} 
a_j4^{-b_j}}\sum_{h=1}^{\infty} 
\omega_1^{m^{1/B(s)} a_j4^{-b_j}\,(h^{b_j}-1)}\\
&\le&
\omega_1^{m^{1/B(s)} a_j4^{-b_j}}
\sum_{h=1}^{\infty} \omega_1^{a_j4^{-b_j}\,(h^{b_j}-1)}\\
&\le&\omega_1^{m^{1/B(s)} a_j4^{-b_j}} R.
\end{eqnarray}
{}From the definition of $m$ we have
$$
\omega_1^{m^{1/B(s)} a_j4^{-b_j}} R
\le \frac{\log(1+\varepsilon^2/4)}{C_j2s}\qquad\mbox{for all}\qquad
j=1,2,\dots,s.
$$
This proves
\begin{equation}\label{eqFnbound}
4nF_n\le 4\left(
-1+\left(1+\frac{\log(1+\e^2/4)}{s}\right)^s\right)
\le
4\left(-1+\exp(\log(1+\e^2/4))\right)=\e^2.
\end{equation}
Now, plugging this into \eqref{linftyapproxerr} 
and taking the square root, we obtain
\begin{equation}\label{linftyapproxerr2}
e^{\bL_{\infty}{\rm -app},
\Lambda^{{\rm std}}}(H(K_{s,\bsa,\bsb}),\cG_{n,s}^\ast) \le \e.
\end{equation}
Hence the first point in \eqref{eqgoal} is shown, and  
it remains to verify that $n$ is of the order
stated in the proposition. We already noted above that $n\le m$. 
However, as pointed out in \cite{KPW12},
\[m=\mathcal{O}\left(\log^{B(s)}\left (1 +\e^{-1}\right)\right),\]
where the factor in the $\mathcal{O}$ notation is independent of
$\e^{-1}$ but dependent on $s$. This completes the proof of 
\eqref{eqgoal}.

Now for the class $\Lambda^{\rm{std}}$,
we conclude from above that
\[n^{\bL_{\infty}\mathrm{-app},
\Lambda^{\rm{std}}}_{\mathrm{abs}}(\varepsilon,s)=\mathcal{O}
\left(\log^{B(s)}\left (1 +\e^{-1}\right)\right).\]
This implies that we indeed have EXP 
for $\Lambda^{\rm{std}}$ for all $\bsa$ and $\bsb$,
with $p(s)=1/B(s)$, and thus
$p^* (s)\ge 1/B(s)$. On the other hand, according to Lemma~\ref{l2linfty}
the rate of exponential convergence for $\bL_{\infty}$-approximation
cannot be larger than for 
$\bL_2$-approximation which was shown 
to be $1/B(s)$ in \cite[Theorem~1, Point~1]{DKPW13}, see Theorem~\ref{thm1in2}.
Thus, we have $p^*(s)= 1/B(s)$.          

We turn to UEXP for the class $\Lambda^{\rm{std}}$.
Suppose that $\bsb$ is such that
\[B=\sum_{j=1}^\infty \frac{1}{b_j}<\infty.\]
Then we can replace $B(s)$ by $B$ in
the above argument,
and we obtain, in exactly the same way,
\[n^{\bL_{\infty}\mathrm{-app},
\Lambda^{\rm{std}}}_{\mathrm{abs}}(\varepsilon,s)=\mathcal{O}
\left(\log^{B}\left (1 +\e^{-1}\right)\right).\]
Hence, we have UEXP with $p^* \ge 1/B$. 
On the other hand, if we have UEXP for $\bL_{\infty}$-approximation, this
implies by Lemma~\ref{l2linfty} 
UEXP for $\bL_2$-approximation,
which in turn, again by 
the results in \cite[Theorem~1, 
Point~2]{DKPW13}, see Theorem~\ref{thm1in2},
implies that $B<\infty$ and that $p^*  \le 1/B$. 

\bigskip

Regarding the class $\Lambda^{\mathrm{all}}$, 
note that we can combine \eqref{errorstdall} and \eqref{errorl2linfty}  to
\begin{equation}\label{eqchainerror}
e^{\bL_2\mathrm{-app},\Lambda^{\mathrm{all}}}(n,s)\le 
e^{\bL_{\infty}\mathrm{-app},
\Lambda^{\mathrm{all}}}(n,s)\le 
e^{\bL_{\infty}\mathrm{-app},\Lambda^{\mathrm{std}}}(n,s).
\end{equation}
We remark that the conditions in Points 1 and 2 of 
Theorem \ref{mainthm} exactly match 
those in \cite[Theorem 1, Points 1 and 2]{DKPW13}.
Hence we can use the results for the class $\Lambda^{\rm std}$
combined with the respective results in \cite{DKPW13} to show
EXP and UEXP for the class $\Lambda^{\mathrm{all}}$. \qed

\section{$\kappa$-EC-weak tractability}\label{seckappa}

In this section we first prove Point 3 of Theorem \ref{mainthm}.
Then we consider the case of $\bL_2$-approximation for $\kappa>1$ since  
this case has not yet been studied.

For $\bL_\infty$-approximation with $\kappa\ge1$, we now prove that 
$\kappa$-EC-WT implies $\lim_ja_j=\infty$. Due to
\eqref{complexityabsnorm} it is enough to consider the class
$\Lambda^{\rm all}$ and the normalized
error criterion. Assume that $\alpha=\sup_ja_j<\infty$.

{}From \eqref{imp2} and the fact that
$\lambda_{s,k}\le1$ for all integer $k$, we have for
$n=n^{\bL_{\infty}\mathrm{-app},\Lambda^{\rm all}}_{\mathrm{norm}}(\e,s)$,
$$
\sum_{k=1}^\infty\lambda_{s,k}-n\le
\sum_{k=n+1}^\infty\lambda_{s,k}\le \e^2\,\sum_{k=1}^\infty\lambda_{s,k}.
$$
Hence,
\begin{equation}\label{eqcomplower}
n\ge(1-\e^2)\,\sum_{k=1}^\infty\lambda_{s,k}=(1-\e^2)\,
\prod_{j=1}^s\left(1+2\sum_{h=1}^\infty\omega^{a_jh^{b_j}}\right)\ge
(1-\e^2)\,\prod_{j=1}^s\left(1+2\omega^{a_j}\right).
\end{equation}
This yields that
\[
\frac{\log\,n}{s+[\log\,\e^{-1}]^\kappa}\ge \frac{\log(1-\e^2)+
\sum_{j=1}^s\log(1+2\omega^{a_j})}{s+[\log\,\e^{-1}]^\kappa}\ge
\frac{\log(1-\e^2)+s\log(1+2\omega^\alpha)}{s+[\log\,\e^{-1}]^\kappa}.
\]
Clearly, for a fixed $\e<1$ and $s$ tending to infinity, the right
hand side of the last formula does not tend to zero. This contradicts
$\kappa$-EC-WT.

We now show that $\lim_ja_j=\infty$ implies $\kappa$-EC-WT.
Due to \eqref{complexityabsnorm} it is enough to consider 
the class $\Lambda^{\rm std}$ and the absolute error
criterion. For any positive $\eta$ we have
$$
\sum_{h=1}^\infty\omega^{\eta\,a_jh^{b_j}}\le
\sum_{h=1}^\infty\omega^{\eta\,a_jh^{b_*}}
\le \omega^{\eta\,a_j}\,\sum_{h=1}^\infty\omega^{\eta\,a_j(h^{b_*}-1)}\le
D_\eta\,\omega^{\eta\,a_j},
$$
where 
$D_\eta=\sum_{h=1}^\infty\omega^{\eta\,a_*(h^{b_*}-1)}<\infty$.
Therefore for any integer $n$ we can estimate
$$
n\lambda^{\eta}_{s,n}\le\sum_{k=1}^\infty\lambda^{\eta}_{s,k}
=\prod_{j=1}^s\left(1+2\sum_{h=1}^\infty\omega^{\eta\,a_jh^{b_j}}\right)
\le \prod_{j=1}^s\left(1+2D_\eta\,\omega^{\eta\,a_j}\right).
$$
Hence, for any positive $\eta$ 
\begin{equation}\label{estlambda}
\lambda_{s,n}\le\frac1{n^{1/\eta}}\,
\prod_{j=1}^s\left(1+2D_\eta\,\omega^{\eta\,a_j}\right)^{1/\eta}
\ \ \ \ \ \mbox{for all}\ \ s,n\in\NN.
\end{equation}
Thus the assumption of Lemma \ref{lemma3} holds with 
$$
\beta_s=1/(2\eta)\ \ \ \ \mbox{and}\ \ \ \ M_s^2=
\prod_{j=1}^s\left(1+2D_\eta\,\omega^{\eta\,a_j}\right)^{1/\eta}.
$$
For $\eta<\tfrac13$ we have $\beta_s>\tfrac32$ and 
$$
e^{\bL_\infty-\mathrm{app},\Lambda^{\rm std}}(n,s)\le 
M_s\left(\frac{\sqrt{2}}2+C(1/(2\eta))\right)\,
\frac1{n^{(1/\eta-3)/2}}.
$$  
Hence 
$$
e^{\bL_\infty-\mathrm{app},\Lambda^{\rm std}}(n,s)\le \e
$$ 
for 
$$
n \ge 1+\left(\left(\frac{\sqrt{2}}2+C(1/(2\eta))\right)\,\frac1{\e}\,
\,\prod_{j=1}^s\left  
(1+2D_\eta\,\omega^{\,\eta a_j}\right)^{1/(2\eta)}\right)^{2\eta/(1-3\eta)}.
$$
and therefore we have 
$$
n_{\mathrm{abs}}^{\bL_{\infty}\mathrm{-app},\Lambda^{\rm all}}(\e,s)
\le 1+\left(\left(\frac{\sqrt{2}}2+C(1/(2\eta))\right)\,\frac1{\e}\,
\,\prod_{j=1}^s\left     
(1+2D_\eta\,\omega^{\,\eta a_j}\right)^{1/(2\eta)}\right)^{2\eta/(1-3\eta)}.
$$
Hence, using $\log (1+x) \le x$ for all $x\ge 0$, we obtain 
$$
\log(n_{\mathrm{abs}}^{\bL_{\infty}\mathrm{-app},\Lambda^{\rm all}}(\e,s)-1)\le
\frac{2\eta}{1-3\eta}\left(\log\left(\frac{\sqrt{2}}2+C(1/(2\eta))\right)
\,+\,\log\,\e^{-1}\right)\,+\, 
\frac{2\,D_\eta}{1-3\eta}\,\sum_{j=1}^s\omega^{\,\eta\,a_j}.
$$

Note that $\lim_j a_j=\infty$ implies that
$\lim_j\omega^{\eta\,a_j}=0$, 
and $\lim_s \sum_{j=1}^s\omega^{\eta\,a_j}/s=0$. Hence
$$
\limsup_{s+\log\varepsilon^{-1}\To\infty}
\frac{\log n_{\mathrm{abs}}^{\bL_{\infty}
\mathrm{-app},\Lambda^{\rm all}}(\e,s)}{s+
[\log\varepsilon^{-1}]^\kappa}\le \frac{2 \eta}{1-3\eta}.
$$
Since $\eta$ can be arbitrarily small, this proves that
$$ 
\lim_{s+\log\varepsilon^{-1}\To\infty}
\frac{\log n_{\mathrm{abs}}^{\bL_{\infty}\mathrm{-app},
\Lambda^{\rm all}}(\e,s)}{s+[\log\varepsilon^{-1}]^\kappa}=0,
$$
and completes the proof of Point 3 of Theorem \ref{mainthm}. 

Point 4 of Theorem \ref{mainthm} 
easily follows by combining Point 2 and Point 3 with
$\kappa=1$. \qed

\medskip

We now turn to $\kappa$-EC-WT for $\bL_2$-approximation. The case
$\kappa=1$ corresponds to EC-WT and is covered in Theorem
\ref{thm1in2} and holds iff $\lim_ja_j=\infty$. We now assume that
$\kappa>1$ and show that the last condition on $\bsa$ is not needed
for the class $\Lambda^{\rm all}$. The case of $\kappa>1$ for 
the class $\Lambda^{\rm std}$ is open. 
\begin{theorem}\label{thmkappawt}
Consider $\bL_2$-approximation defined over the Korobov
space $H(K_{s,\bsa,\bsb})$ with weight sequences $\bsa$ and $\bsb$
satisfying~\eqref{coefficients} and the class $\Lambda^{\rm all}$.
Then for $\kappa>1$ 
\begin{center}
$\bL_2$-approximation is $\kappa$-EC-WT for all considered $\bsa$ and $\bsb$.
\end{center}
\qed
\end{theorem}
\begin{proof}
{}From \eqref{estlambda} we conclude that 
$[e^{\bL_2-\mathrm{app},\Lambda^{{\rm all}}}(n,s)]^2=\lambda_{s,n+1}\le \e^2$ 
for
$$
n \ge \frac{(1+2D_\eta\,\omega^{\,\eta a_*})^s}{\e^{2\eta}}
$$
and hence 
$$
n^{\bL_2\mathrm{-app},\Lambda^{\rm all}}(\varepsilon,s) \le \frac{(1+2D_\eta\,\omega^{\,\eta a_*})^s}{\e^{2\eta}}.
$$
Therefore for any positive~$\eta$ we have
$$
\log\,
  n^{\bL_2\mathrm{-app},\Lambda^{\rm all}}(\varepsilon,s)\le
  2\eta\,\log\,\varepsilon^{-1}
 +2s\,D_\eta \omega^{\,\eta a_*}. 
$$ 
Hence 
$$
\frac{\log\,
  n^{\bL_2\mathrm{-app},\Lambda^{\rm all}}(\varepsilon,s)}
{s+[\log\,\e^{-1}]^\kappa}
\le
  \frac{2\eta\,\log\,\varepsilon^{-1}}{s+[\log\,\e^{-1}]^\kappa}+
\frac{2s\,D_\eta\,\omega^{\,\eta\,a_*}}{s+[\log\,\e^{-1}]^\kappa}\le
  \frac{2\eta\,\log\,\varepsilon^{-1}}{s+[\log\,\e^{-1}]^\kappa}+
2\,D_\eta\,\omega^{\,\eta\,a_*}.
$$
The first term of the last bound goes to zero as
$s+\log\,\varepsilon^{-1}$ goes to infinity since $\kappa>1$,
whereas the second term
is arbitrarily small for large $\eta$. Therefore
$$
\lim_{s+\log\,\varepsilon^{-1}\to \infty} \frac{\log\,
  n^{\bL_2\mathrm{-app},\Lambda^{\rm all}}
(\varepsilon,s)}{s+[\log\,\e^{-1}]^\kappa}=0.
$$
This means that $\kappa$-EC-WT holds, for $\kappa>1$, 
for all considered $\bsa$ and $\bsb$.
\end{proof}

\section{EC-(strong) polynomial tractability}\label{secpt}
We now prove Points 5 and 6 of Theorem \ref{mainthm}. For this we need
the following proposition. 

\begin{proposition}\label{prop_std_spt}
Assume that
$$
B:=\sum_{j=1}^\infty\frac1{b_j}<\infty\quad\mbox{and}\quad
\alpha^*:=\liminf_{j\to\infty}\ \frac{\log\,a_j}j>0.
$$
Let
$\cG_{n,s}^{\ast}$ be a regular grid with mesh-sizes $m_1,m_2,\ldots,m_s$
given by
$$
m_j=2\,\left\lceil\,
    \left(\frac{\log\,\e^{-2}}{a_j^{\beta}\log\,
\omega^{-1}}\right)^{1/b_j}\right
\rceil\,-\,1\qquad\mbox{for all}\qquad j=1,2,\dots,s,
$$
with $\beta\in (0,1)$.

Then for any $\eta \in (0,\min(a_*^{1-\beta},1))$ and any $\delta \in
(0,\alpha^*)$ 
there exists a positive $\overline{C}_{\beta,\delta,\eta}$ 
such that 
$$
e^{\bL_{\infty}{\rm -app}}(H(K_{s,\bsa,\bsb}),
\cG_{n,s}^\ast)\le \overline{C}_{\beta,\delta,\eta}\ 
\e^{\min(a_*^{1-\beta},1)-\eta}
$$
and
$$
n=\mathcal{O}\left(\left(1+\log\,\e^{-1}
\right)^{B+(\log 3)/(\beta\,\delta)}\right),
$$
with the factor in the $\mathcal{O}$ notation independent
of $\e^{-1}$ and $s$, and dependent only on $\beta$ and~$\delta$.
\end{proposition}
\begin{proof}
We first note that $m_j\ge1$ and is always an odd number. Furthermore $m_j=1$
iff $a_j\ge ((\log \varepsilon^{-2})/(\log \omega^{-1}))^{1/\beta}$.
Since for all $\delta\in(0,\alpha^*)$ there exists an integer
$j^*_\delta$ such that 
$$
a_j\ge \exp(\delta j) \quad\mbox{for all}\quad j\ge j^*_\delta,
$$
we conclude that
$$
j\ge j^{*}_{\beta,\delta}:=\max\left(j^*_\delta,
\frac{\log(((\log \varepsilon^{-2})/(\log\omega^{-1}))^{1/\beta})}
{\delta}\right)
\ \ \ \mbox{implies}\ \ \ m_j=1.
$$

{}From \eqref{eqdoublesum} we know that 
$$
\left[e^{\bL_{\infty}{\rm -app}}(H(K_{s,\bsa,\bsb}),\cG_{n,s}^*)\right]^2 \le
  4\sum_{\bsv\in \calV_n}\sum_{\bsl\in \cG_{n,s}^{* \bot}\setminus
\{\bszero\}}\omega_{\bsv+\bsl}.
$$
We now consider
$$
\sum_{\bsl \in \mathcal{G}^{* \bot}_{n,s} \setminus \{\bszero\} }
\omega_{\bsv+\bsl}=
\sum_{\emptyset \neq \uu \subseteq\{1,\ldots, s\}}
\prod_{j\in \uu}\left(
\sum_{h_j \in \mathbb{Z}\setminus \{0\}}\omega^{a_j|v_j+m_j h_j|^{b_j}}
\right)\, \prod_{j \not \in \uu}
\omega^{a_j|v_j|^{b_j}},
$$
where we separated the cases for $h_j \in \ZZ\setminus \{0\}$
and $h_j=0$.
We bound the second product by one such that
$$
\sum_{\bsl \in \mathcal{G}^{* \bot}_{n,s} \setminus \{\bszero\} }
\omega_{\bsv+\bsl}\le
\sum_{\emptyset \neq \uu \subseteq\{1,\ldots, s\}}
\prod_{j\in \uu}\left(
\sum_{h \in \mathbb{Z}\setminus \{0\}}\omega^{a_j|v_j+m_j h|^{b_j}}
\right).
$$
Note that for $\bsv\in\calV_n$ we have from \eqref{eqVn}
that $\abs{v_j}< (m_j +1)/2$ for $j=1,2,\ldots,s$.
In particular, if $m_j=1$ then $v_j=0$ and
\begin{equation}\label{ineqhone}
\sum_{h\in \mathbb{Z}\setminus \{0\}}\omega^{a_j|v_j+m_j
  h|^{b_j}}=2\sum_{h=1}^\infty\omega^{a_jh^{b_j}}\le
2\sum_{h=1}^\infty\omega^{a_jh^{b_\ast}}=2\omega^{a_j} 
\sum_{h=1}^\infty\omega^{a_j(h^{b_\ast}-1)}\le 2\omega^{a_j} D,
\end{equation}
where $D:=D_1=\sum_{h=1}^\infty \omega^{a_\ast (h^{b_\ast}-1)}$.

Let $m_j\ge3$. Since
$|v_j|<(m_j+1)/2$, 
we conclude that $|v_j|\le (m_j+1)/2-1=(m_j-1)/2$, and
$h\not=0$ implies
$$
|v_j+m_j h|\ge m_j|h|-|v_j|\ge
\frac{m_j+1}{2} |h|.
$$
Therefore
\begin{eqnarray}
\sum_{h \in \mathbb{Z}\setminus \{0\}}\omega^{a_j|v_j+m_j
  h|^{b_j}}&\le&
2\sum_{h=1}^\infty\omega^{a_j[(m_j+1)/2]^{b_j}h^{b_j}} \nonumber\\
&=& 2\omega^{a_j[(m_j+1)/2]^{b_j}} 
\sum_{h=1}^\infty\omega^{a_j[(m_j+1)/2]^{b_j}(h^{b_j}-1)}\nonumber\\
&\le & 2\omega^{a_j[(m_j+1)/2]^{b_j}} 
\sum_{h=1}^\infty\omega^{a_\ast (h^{b_\ast}-1)}\nonumber\\
&=& 2 \omega^{a_j[(m_j+1)/2]^{b_j}} D.\label{ineqhtwo}
\end{eqnarray}
The inequalities \eqref{ineqhone} and \eqref{ineqhtwo} can be combined as
$$
\beta_j:=\sum_{h \in \mathbb{Z}\setminus \{0\}}\omega^{a_j|v_j+m_j
  h|^{b_j}}\le
2 \omega^{a_j[(m_j+1)/2]^{b_j}} D.
$$
Note that
$$
\sum_{\emptyset \neq \uu \subseteq\{1,\ldots, s\}}
\prod_{j\in \uu}\left(\sum_{h \in \mathbb{Z}\setminus \{0\}}
\omega^{a_j |v_j+ m_j h|^{b_j}}\right)
=-1+
\sum_{\uu \subseteq\{1,\ldots, s\}}
\prod_{j\in \uu} \beta_j=-1+\prod_{j=1}^s(1+\beta_j).
$$
Consequently,
\begin{eqnarray*}
\left[e^{\bL_{\infty}{\rm -app}}(H(K_{s,\bsa,\bsb}),
\cG_{n,s}^\ast)\right]^2 & \le & 4 \sum_{\bsv \in \calV_n} \left[ -1 + \prod_{j=1}^s
\left(1 + 2\,\omega^{\,a_j[(m_j+1)/2]^{b_j}} D
\right)\right]\\
& = & 4 n \left[ -1 + \prod_{j=1}^s
\left(1 + 2\,\omega^{\,a_j[(m_j+1)/2]^{b_j}} D
\right)\right].
\end{eqnarray*}
Using $\log (1+x) \le x$ we obtain
$$
\log\left[ \prod_{j=1}^s
\left(1 + 2\,\omega^{\,a_j[(m_j+1)/2]^{b_j}} D
\right)\right] \le
2D \sum_{j=1}^s\omega^{\,a_j[(m_j+1)/2]^{b_j}}.
$$
{}From the definition of $m_j$ we have
$a_j[(m_j+1)/2]^{b_j}\ge a_j^{1-\beta}\,(\log\,\e^{-2})/\log\,\omega^{-1}$. Therefore,
$$
\omega^{a_j[(m_j+1)/2]^{b_j}}\le
\omega^{a_j^{1-\beta}\,(\log\,\e^{-2})/\log\,\omega^{-1}}=
\e^{2\,a_j^{1-\beta}}.
$$
Since $a_j\ge a_*$ 
for $j\le j^*_{\beta,\delta}-1$ and
$a_j\ge \exp(\delta j)$ for $j\ge j^*_{\beta,\delta}$
we obtain
\begin{eqnarray*}
\gamma&:=& 2D\sum_{j=1}^s\omega^{\,a_j[(m_j+1)/2)]^{b_j}}\le
2D\,\left((j^*_{\beta,\delta}-1)\e^{2a_*^{1-\beta}}
+\e^2\,
\sum_{j=j^*_{\beta,\delta}}^\infty\e^{2[\exp((1-\beta) \delta j)-1]}\right)\\
& \le & 2D\e^{2\min(a_*^{1-\beta},1)}\, 
\left(j^*_{\beta,\delta}-1 + 
\sum_{j=j^*_{\beta,\delta}}^\infty 
\e^{2[\exp((1-\beta) \delta j)-1]}\right).
\end{eqnarray*}

Without loss of generality, we now choose $\e$ such that $\e^{-2}\ge 2$.
Then,
$$
2D\, \left(j^*_{\beta,\delta}-1 + 
\sum_{j=j^*_{\beta,\delta}}^\infty \e^{2[\exp((1-\beta) \delta j)-1]}\right)
\le C_{\beta,\delta},
$$
where
$$
C_{\beta,\delta}:=2D\,\left(j^*_{\beta,\delta}
-1+\sum_{j=j^*_{\beta,\delta}}^\infty
\left(\frac{1}{2}\right)^{\exp((1-\beta) \delta j)-1}\right)<\infty.
$$

Hence we obtain 
$\gamma\le C_{\beta,\delta}\e^{2\min(a_*^{1-\beta},1)}$, 
and by choosing, without loss of generality, 
$\e^{-2\min(a_*^{1-\beta},1)}\ge C_{\beta,\delta}$, we have $\gamma\le 1$.

Using convexity we easily check that
$-1+\exp(\gamma) \le (\mathrm{e}-1)\gamma$ for all $\gamma\in[0,1]$. Thus
for $\e^{-2\min(a_*^{1-\beta},1)}\ge C_{\beta,\delta}$ we obtain
\begin{eqnarray*}
-1 + \prod_{j=1}^s
\left(1 +
2\,\omega^{\,a_j[(m_j+1)/2]^{b_j}} D\right)
& \le & -1 + \exp\left(2D
\sum_{j=1}^s \omega^{\,a_j[(m_j+1)/2]^{b_j}}\right)\\
&=&-1+\exp(\gamma)\le (\mathrm{e}-1)\gamma\\
& \le & C_{\beta,\delta}\,(\mathrm{e}-1)\e^{2\min(a_*^{1-\beta},1)}
\end{eqnarray*}
and hence
\begin{eqnarray}\label{bd1}
e^{\bL_{\infty}{\rm -app}}(H(K_{s,\bsa,\bsb}),\cG_{n,s}^\ast) 
\le 2 \sqrt{n C_{\beta,\delta}\,(\mathrm{e}-1)}\ \e^{\min(a_*^{1-\beta},1)}.
\end{eqnarray}
We now estimate the number $n$ of function values used by the
algorithm. We have
$$
n=\prod_{j=1}^sm_j=\prod_{j=1}^{\min(s,j^*_{\beta,\delta})}m_j\le
\prod_{j=1}^{j^*_{\beta,\delta}}
\left(1+2\left(\frac{\log\,\e^{-2}}{a_j^{\beta}\,
\log\,\omega^{-1}}\right)^{1/b_j}\right).
$$
We bound $j^*_{\beta,\delta}$ by the sum of the two terms defining it,
and obtain                    
\begin{eqnarray*}
n&\le& 3^{j^*_{\beta,\delta}} a_\ast^{-B\beta}\,
\left(\frac{\log\,\e^{-2}}{\log\,\omega^{-1}}\right)^B
\le
 3^{j_{\delta}^*} a_\ast^{-B\beta}\left(\frac{\log\,\e^{-2}}{\log\,\omega^{-1}}
\right)^{B+(\log 3)/(\beta\,\delta)}\\
&=&\mathcal{O}\left(\left(1+\log\,\e^{-1}
\right)^{B+(\log 3)/(\beta\,\delta)}\right).
\end{eqnarray*}
Inserting this into \eqref{bd1} we obtain for any $\eta>0$ that
\begin{eqnarray*}
e^{\bL_{\infty}{\rm -app}}(H(K_{s,\bsa,\bsb}),\cG_{n,s}^\ast) \le 
\overline{C}_{\beta,\delta,\eta}\ \e^{\min(1,a_*^{1-\beta})-\eta},
\end{eqnarray*}
where the positive quantity $\overline{C}_{\beta,\delta,\eta}$ 
depends on $\beta, \delta$ and $\eta$, but not on $\varepsilon^{-1}$ and $s$.
This completes the proof of the proposition.
\end{proof}

We are ready to prove Points 5 and 6 of Theorem \ref{mainthm}. We
consider four cases depending on the information class and the error
criterion.
\begin{itemize}
\item Case 1: $\Lambda^{\rm std}$ and the absolute error criterion. 

We already showed that EC-PT implies EC-PT + EXP and EC-PT +UEXP. 
Therefore  the chain of implications from EC-SPT+UEXP to EC-PT is
trivial.

Hence, it is enough to show that EC-PT implies EC-SPT+UEXP.  Note that 
EC-PT for $\bL_{\infty}$-approximation 
implies by Lemma~\ref{l2linfty} EC-PT for
$\bL_2$-approximation which in turn 
by \cite[Theorem~1, Point 5]{DKPW13}, 
see also Theorem~\ref{thm1in2},
implies EC-SPT+UEXP for $\bL_2$-approximation.
This, however, by  \cite[Theorem~1, Point 6]{DKPW13} implies that $B<\infty$ and
$\alpha^*>0$, where $\alpha^*$ is defined as in \eqref{eqalphastar}. 
We will show below that these conditions on $\bsa$ and $\bsb$ 
imply EC-SPT+UEXP for $\bL_{\infty}$-approximation.
This ends the proof of Point 5 for this case.

We now prove Point 6. 
The necessity of the conditions for EC-SPT+UEXP on $\bsa$ and~$\bsb$
for $\bL_{\infty}$-approximation and the class~$\Lambda^{\rm{std}}$
follows from the same conditions for $\bL_2$-approximation shown 
in \cite[Theorem~1, Point~6]{DKPW13},
and the fact that the information complexity for the $\bL_{\infty}$-case cannot
be smaller than for the $\bL_2$-case.

The sufficiency of the conditions is shown by the use of 
Proposition~\ref{prop_std_spt}, 
under the assumption of \eqref{eqalphastar}, which states that
$$
n^{L_\infty\mathrm{-app},\Lambda^{\rm std}}_{\mathrm{abs}}(
\overline{C}_{\beta,\delta,\eta}\,\e^{\min(a_*^{1-\beta},1)-\eta},s)
=\mathcal{O}\left((1+\log\,\e^{-1})\right)^{B+(\log\,3)/(\beta\delta)}.
$$
By replacing 
$\overline{C}_{\beta,\delta,\eta}\,\e^{\min(a_*^{1-\beta},1)-\eta}$ by
$\e$ we obtain 
\begin{eqnarray*}
n^{L_\infty\mathrm{-app},\Lambda^{\rm std}}_{\mathrm{abs}}(\e,s)
&=&\mathcal{O}\left(1+\log\,
[\overline{C}_{\beta,\delta,\eta}\,\e^{\min(a_*^{1-\beta},1)-\eta}]^{-1}
\right)^{B+(\log\,3)/(\beta\delta)}\\
&=&\mathcal{O}\left(1+\log\,\e^{-1}\right)^{B+(\log\,3)/(\beta\delta)}
\end{eqnarray*}
with the factor in the $\mathcal{O}$ notation independent of
$\e^{-1}$ and $s$. This proves EC-SPT+UEXP with exponent
$$
\tau=B+\frac{\log 3}{\beta\,\delta}.
$$
Since $\beta$ can be arbitrarily close to one, and $\delta$ can be
arbitrarily close to $\alpha^*$,
the exponent $\tau^*$ of EC-SPT is at most
$$
B+\frac{\log 3}{\alpha^*},
$$
where for $\alpha^*=\infty$ we have $\frac{\log 3}{\alpha^*} = 0$.
This completes the proof of Point 6 
for $\Lambda^{\rm std}$ and the absolute error criterion.

\item Case 2: $\Lambda^{\rm std}$ and the normalized error criterion. 

To prove Point 5, it is clear 
that EC-SPT+UEXP implies EC-PT. Let us now assume we have EC-PT. 
Then we also have EC-PT for $\Lambda^{\rm all}$ and the 
normalized error criterion. This, by what we will show below,
implies \eqref{eqalphastar}. 
As we know, \eqref{eqalphastar} implies EC-SPT+UEXP 
for $\Lambda^{\rm std}$ and the absolute error criterion. 
However, by \eqref{complexityabsnorm}, the latter implies EC-SPT+UEXP 
for $\Lambda^{\rm std}$ and the normalized error criterion.  

To prove Point 6, the sufficiency of the conditions 
follows from the corresponding results 
for $\Lambda^{\rm std}$ and the absolute error criterion. 
The necessary conditions for $\Lambda^{\rm std}$ follow 
from the necessary conditions for $\Lambda^{\rm all}$ 
and the normalized error criterion that we will prove below.

\item Case 3: $\Lambda^{\rm all}$ and the absolute error criterion.

Let us again start with Point 5. As before, it is enough to show
that EC-PT implies EC-SPT+UEXP. EC-PT for $\bL_\infty$-approximation and
$\Lambda^{\rm all}$ implies
EC-PT for $\bL_2$-approximation for $\Lambda^{\rm all}$. 
Then it follows from~\cite[Theorem~1, Points~5 and~6]{DKPW13}
that \eqref{eqalphastar} holds. This condition, however, 
implies EC-SPT+UEXP for $\Lambda^{\rm std}$ and 
the absolute error criterion, and hence also 
EC-SPT+UEXP for $\Lambda^{\rm all}$ and the absolute error criterion. 
Point 5 is therefore shown.

For Point 6, the sufficient conditions for EC-SPT+UEXP 
follow from \eqref{complexitystdall},
and from the results for $\Lambda^{\rm std}$ and the 
absolute error criterion. On the other hand, 
the necessary conditions for EC-SPT+UEXP follow 
from \eqref{complexityl2linfty}, and
from the results for $\bL_2$-approximation in~\cite[Theorem~1, Point~6]{DKPW13}.

\item Case 4: $\Lambda^{\rm all}$ and the normalized error criterion.

Let us start with Point 5. 
Again it is obvious that EC-SPT+UEXP implies EC-PT. 
Conversely, assume now that we have EC-PT. 
Then by \eqref{eqcomplower}, we obtain 
for $n=n_{\rm norm}^{\bL_{\infty}{\rm -app},\Lambda^{\rm all}}(\varepsilon,s)$,
\[
 n\ge (1-\varepsilon^2)\prod_{j=1}^s (1+2\omega^{a_j}).
\]
Since we assumed EC-PT, this means that 
$\prod_{j=1}^s (1+2\omega^{a_j})$ may at most depend polynomially on
$s$. 
However, due to results from \cite{SW01}, this can only happen if 
\begin{equation}\label{eqcondpt}
 \limsup_{s\To\infty} \sum_{j=1}^s \omega^{a_j}/\log s <\infty.
\end{equation}
So, let us assume that \eqref{eqcondpt} is fulfilled.

Next, consider the square of the initial error,
\begin{eqnarray*}
 \left[e^{\bL_{\infty}\mathrm{-app}}(0,s)\right]^2&=&
\prod_{j=1}^s\left(1+2\sum_{h=1}^\infty \omega^{a_j h^{b_j}}\right)
 \le \prod_{j=1}^s\left(1+2\sum_{h=1}^\infty \omega^{a_j h^{b_\ast}}\right)\\
 &=& \prod_{j=1}^s\left(1+2\omega^{a_j}
\sum_{h=1}^\infty \omega^{a_j (h^{b_\ast}-1)}\right)
 \le \prod_{j=1}^s\left(1+2\omega^{a_j}
\sum_{h=1}^\infty \omega^{a_\ast (h^{b_\ast}-1)}\right)\\
&=&\prod_{j=1}^s\left(1+\omega^{a_j}A\right),
\end{eqnarray*}
where $A:=2D_1= 2\sum_{h=1}^\infty \omega^{a_\ast
  (h^{b_\ast}-1)}<\infty$. 
Due to \eqref{eqcondpt} we see 
that $e^{\bL_{\infty}\mathrm{-app}}(0,s)$ is bounded 
by an expression that depends at most polynomially on $s$. 
Hence it follows that the conditions for 
EC-PT regarding the normalized and the absolute error criteria are 
equivalent. For the absolute error criterion, we already 
know that EC-PT implies \eqref{eqalphastar}. This implies EC-SPT+UEXP 
due to Point  6 that we show below.

Let us come to Point 6. Suppose that we have EC-SPT+UEXP. 
This implies EC-PT, which, by the previous argument implies 
\eqref{eqalphastar}. 

Suppose now that \eqref{eqalphastar} holds. 
Then we know from above that EC-SPT+UEXP 
for $\Lambda^{\rm all}$ and the absolute 
error criterion holds. 
This implies EC-SPT+UEXP for $\Lambda^{\rm all}$ and the normalized 
error criterion. \qed  
\end{itemize}

\section{Comparison of $\bL_\infty$- and $\bL_2$-approximation}\label{secend}

We briefly compare the results for $\bL_\infty$- and
$\bL_2$-approximation. As before,
$$
B=\sum_{j=1}^\infty \frac{1}{b_j}\ \ \ \  \mbox{and}\ \ \ \ 
\alpha^*=\liminf_{j\rightarrow \infty}\, \frac{\log a_j}{j}.
$$                                       
Unless noted otherwise, the conditions in Table~\ref{tablecomp} 
are valid for $\Lambda^{{\rm all}}$ 
and $\Lambda^{{\rm std}}$ and, in the $\bL_{\infty}$-case, 
for both error criteria.

\begin{table}
\begin{center}
\renewcommand*\arraystretch{2}
\begin{tabular}{|l|l|l|}
\hline
Property & conditions ($\bL_\infty$) & conditions ($\bL_2$)\\
\hline \hline
EXP & for all considered $\bsa$ and $\bsb$ & 
for all considered $\bsa$ and $\bsb$\\
\hline
UEXP & iff $\bsb$ such that $B< \infty$ & 
iff $\bsb$ such that $B< \infty$\\
\hline
$\kappa$-EC-WT, $\kappa>1$ for $\Lambda^{{\rm all}}$ & iff 
$\lim_j a_j=\infty$ & for all considered $\bsa$ and $\bsb$\\
\hline
$\kappa$-EC-WT, $\kappa>1$ for $\Lambda^{{\rm std}}$ & iff 
$\lim_j a_j=\infty$ & open 
\\
\hline
EC-WT & iff 
$\lim_j a_j=\infty$ & iff $\lim_j a_j=\infty$ \\
\hline
EC-PT & iff EC-SPT & iff EC-SPT\\
\hline
EC-SPT & iff $B< \infty$ and $\alpha^* >0$ & iff 
$B< \infty$ and $\alpha^* >0$ \\
\hline
\end{tabular}
\caption{Comparison of results for $\bL_\infty$- and $\bL_2$-approximation}\label{tablecomp}
\end{center}
\end{table}

We see that the only difference between $\bL_{\infty}$- 
and $\bL_2$-approximation is 
for the property $\kappa$-EC-WT for $\kappa>1$ 
for the information class $\Lambda^{{\rm all}}$. 
The condition for  $\kappa$-EC-WT for $\kappa>1$ 
for $\bL_2$-approximation and the information class 
$\Lambda^{{\rm std}}$ remains an open question.

\section{Remarks on $\bL_p$-approximation}\label{rem_lp_approx}

Let us, finally, briefly comment on the case of 
$\bL_p$-approximation for $p\in [2,\infty]$. 
Let us consider $\bL_p$-approximation of functions in 
$H(K_{s,\bsa,\bsb})$, and the absolute error criterion. 
Let $e^{\bL_p-{\rm app},\Lambda}(n,s)$ 
denote the $n$th minimal worst case error, 
and let $n_{\rm abs}^{\bL_p-{\rm app},\Lambda}(\e , s)$ be
the information complexity of this problem. 

Then, similarly to the proof of Lemma \ref{l2linfty}, we see that
\[
 e^{\bL_2-{\rm app},\Lambda}(n,s)
\le e^{\bL_p-{\rm app},\Lambda}(n,s)
\le e^{\bL_\infty-{\rm app},\Lambda}(n,s)\ \ \ \ 
\mbox{for all}\ \ n,s\in\NN,
\]
and
\[
 n^{\bL_2-{\rm app},\Lambda}(\e , s) 
\le n_{\rm abs}^{\bL_p-{\rm app},\Lambda}(\e , s)
\le n_{\rm abs}^{\bL_\infty-{\rm app},\Lambda}(\e , s)
\ \ \ \ \mbox{for all}\ \ \e\in (0,1),\ s\in\NN.
\]
Hence, we can conclude that for all situations mentioned in 
Table \ref{tablecomp}, except for $\kappa$-EC-WT with $\kappa>1$, 
the results for $\bL_p$-approximation and the absolute error criterion 
are the same as those for $\bL_2$-approximation and $\bL_\infty$-approximation. 
Whether a similar observation is also true 
for the normalized error criterion and for $p\in[1,2)$ remain
an open question.

\begin{small}
\noindent\textbf{Authors' addresses:}
\\ \\
\noindent Peter Kritzer,
\\
Johann Radon Institute for Computational and Applied Mathematics (RICAM),
Austrian Academy of Sciences, Altenbergerstr.~69, 4040 Linz, Austria\\
 \\
\noindent Friedrich Pillichshammer,
\\
Department of Financial Mathematics,
Johannes Kepler University Linz, Altenbergerstr.~69, 4040 Linz, Austria\\
 \\
\noindent Henryk Wo\'{z}niakowski, \\
Department of Computer Science, Columbia University, New York 10027,
USA and Institute of Applied Mathematics,
University of Warsaw, ul. Banacha 2, 02-097 Warszawa, Poland\\ \\

\noindent \textbf{E-mail:} \\
\texttt{peter.kritzer@oeaw.ac.at}\\
\texttt{friedrich.pillichshammer@jku.at} \\
\texttt{henryk@cs.columbia.edu}
\end{small}

\begin{thebibliography}{99}

\bibitem{Aron} N. Aronszajn. Theory of reproducing kernels.
Trans. Amer. Math. Soc. 68, 337--404, 1950.

\bibitem{DKPW13} J.~Dick, P.~Kritzer, F.~Pillichshammer,
H.~Wo\'{z}niakowski. Approximation of analytic functions in Korobov spaces. J. Complexity 30, 2--28, 2014.

\bibitem{DLPW11} J.~Dick, G.~Larcher, F.~Pillichshammer,
H.~Wo\'{z}niakowski. Exponential convergence and tractability of
multivariate integration for Korobov spaces. Math. Comp. 80,
905--930, 2011.

\bibitem{IKPW15} C.~Irrgeher, P.~Kritzer, 
F.~Pillichshammer, H.~Wo\'{z}niakowski. 
Tractability of Multivariate Approximation defined 
over Hilbert spaces with Exponential Weights. Submitted, 2015.  
See \url{arXiv:1502.03286}

\bibitem{KPW12} P. Kritzer, F. Pillichshammer, H. Wo\'{z}niakowski.
Multivariate integration
of infinitely many times differentiable functions
in weighted Korobov spaces. Math. Comp. 83, 1189--1206, 2014.

\bibitem{KPW14} P. Kritzer, F. Pillichshammer, H. Wo\'{z}niakowski.
Tractability of multivariate analytic problems. 
In P. Kritzer, H. Niederreiter,
F. Pillichshammer, A. Winterhof (eds.) 
\textit{Uniform Distribution and Quasi-Monte Carlo Methods. Discrepancy,
Integration and Applications}, De Gruyter, Berlin, 124--170, 2014.

\bibitem{KSW06} F.Y. Kuo, I.H. Sloan, H. Wo\'{z}niakowski. Lattice
rules for multivariate approximation in the worst case setting.
In: H. Niederreiter, D. Talay (eds.). \textit{Monte Carlo and
Quasi-Monte Carlo Methods 2004}. Springer, Berlin, pp. 289--330, 2006.

\bibitem{KWW08} F.Y. Kuo, G.W.~Wasilkowski, H. Wo\'{z}niakowski.
Multivariate $L_{\infty}$ approximation in the worst case setting over reproducing kernel Hilbert spaces.
J. Approx. Theory 152, 135--160, 2008.

\bibitem{KWW09a} F.Y. Kuo, G.W.~Wasilkowski, H. Wo\'{z}niakowski.
On the power of standard information for multivariate approximation in the worst case setting. J. Approx. Theory 158, 97--125, 2009.

\bibitem{KWW09b} F.Y. Kuo, G.W.~Wasilkowski, H. Wo\'{z}niakowski.
Lattice algorithms for multivariate $L_{\infty}$ approximation 
in the worst-case setting. Constr. Approx. 30, 475--493, 2009.


\bibitem{LH03} D. Li, F.J. Hickernell. Trigonometric spectral collocation
methods on lattices. Contemp. Math. 330, 121--132, 2003.

\bibitem{NSW04} E.~Novak, I.H.~Sloan, H.~Wo\'{z}niakowski. Tractability of
approximation for weighed Korobov spaces on
classical and quantum computers. Found. Comput. Math. 4, 121--156, 2004.


\bibitem{NW08}
E.~Novak, H.~Wo\'zniakowski. \textit{Tractability of
Multivariate Problems, Volume I: Linear
  Information}. EMS, Z\"urich, 2008.

\bibitem{NW10}
E.~Novak, H.~Wo\'zniakowski. \textit{Tractability of
Multivariate Problems, Volume II: Standard Informations for
Functionals}. EMS, Z\"urich, 2010.

\bibitem{NW12}
E.~Novak, H.~Wo\'zniakowski. \textit{Tractability of
Multivariate Problems, Volume III: Standard Informations for
Operators}. EMS, Z\"urich, 2012.

\bibitem{PP13}
I. Petras, A. Papageorgiou. 
A new criterion for tractability of multivariate problems. 
J. Complexity 30, 605--619, 2014.

\bibitem{SW01}
I.H.~Sloan, H.~Wo\'{z}niakowski. 
Tractability of multivariate integration for weighted 
Korobov classes. J. Complexity 17, 697--721, 2001.


\bibitem{TWW88}
J.F.~Traub, G.W.~Wasilkowski, H.~Wo\'{z}niakowski. \newblock
{\em Information-Based Complexity}. \newblock Academic Press, New
York, 1988.

\bibitem{ZKH09} X. Zeng, P. Kritzer, F.J. Hickernell. Spline
methods using integration lattices and digital nets. Constr.
Approx. 30, 529--555, 2009.
\end{thebibliography}
\end{document}